\documentclass[9pt,dvipsnames]{amsart}
\usepackage{amsmath, amsthm, amscd, amsfonts,graphicx,mathtools}
\usepackage{amsmath,amsthm,amsthm, amscd, amsfonts, stackrel, latexsym,amssymb,mathrsfs,textcomp,wasysym,hyperref}
\usepackage{enumitem}
\usepackage{tikz}
\usepackage{tikz-cd}
\usetikzlibrary{matrix,arrows}
\usepackage[all]{xy}

\usepackage{color} 
\definecolor{slateblue}{rgb}{0,0.0,0.8}

\usepackage{xcolor}
\usepackage{graphicx}
\hypersetup{%
	colorlinks=true,
	linkcolor=slateblue,
	linkbordercolor=gray,
	citecolor=slateblue
}

\theoremstyle{definition}
\newtheorem{definition}{Definition}[section]
\newtheorem{example}[definition]{Example}
\theoremstyle{remark}
\newtheorem{remark}[definition]{Remark}
 
\theoremstyle{plain}
\newtheorem{theorem}[definition]{Theorem}
\newtheorem{lemma}[definition]{Lemma}
\newtheorem{proposition}[definition]{Proposition}
\newtheorem{corollary}[definition]{Corollary}
\newtheorem{notation}[definition]{Notation}

\makeatletter
\@namedef{subjclassname@2020}{\textup{2020} Mathematics Subject Classification}
\makeatother

\usepackage{amsthm}

\newtheorem{innercustomgeneric}{\customgenericname}
\providecommand{\customgenericname}{}
\newcommand{\newcustomtheorem}[2]{%
	\newenvironment{#1}[1]
	{%
		\renewcommand\customgenericname{#2}%
		\renewcommand\theinnercustomgeneric{##1}%
		\innercustomgeneric
	}
	{\endinnercustomgeneric}
}

\newcustomtheorem{customthm}{Theorem}
\newcustomtheorem{customlemma}{Lemma}

\DeclareMathOperator{\Sym}{Sym}

\title[Diffeological Tangent Spaces and Distributional Linearization for Lifted Euler--Reynolds Limits]{Diffeological Tangent Spaces and Distributional Linearization for Lifted Euler--Reynolds Limits}
\author{Alireza Ahmadi}
\address{Alireza Ahmadi, Department of Mathematical Sciences, Yazd University, 89195--741, Yazd, Iran}
\email{ahmadi@stu.yazd.ac.ir}

\author{Jean-Pierre Magnot}
\address{Jean-Pierre Magnot,  CNRS, LAREMA, SFR MATHSTIC, F-49000 Angers, France \\ and \\  Lyc\'ee Jeanne d'Arc,  Avenue de Grande Bretagne,  63000 Clermont-Ferrand, France}

\email{jean-pierr.magnot@ac-clermont.fr}

\author{Bijan Davvaz}
\address{Bijan Davvaz, Department of Mathematical Sciences, Yazd University, 89195--741, Yazd, Iran}
\email{davvaz@yazd.ac.ir}

 
\begin{document}

\maketitle
\begin{abstract}
	We develop a diffeological framework for the geometry of
	Euler--Reynolds subsolutions of the incompressible Euler equations.
	Passing to a lifted formulation in which the velocity, quadratic flux,
	and trace-free Reynolds stress are treated as independent variables, we
	construct a diffeological limit space obtained as the weak closure of
	smooth strict subsolutions. Its internal tangent spaces provide an
	intrinsic notion of infinitesimal deformation despite the absence of any
	underlying manifold structure.
	We prove that ambient realizations of internal tangent vectors satisfy
	the linearized Euler--Reynolds equations in the sense of distributions,
	thereby establishing a natural first-order theory for lifted weak limit
	spaces. We further describe the tangent directions compatible with
	the Euler locus and identify the kernels of the natural velocity, flux,
	and full-state observables. These results distinguish observable
	perturbations from hidden stress-gauge directions that encode
	infinitesimal variations of the Reynolds stress.
	Finally, we introduce a finite-mixture model for lifted
	Euler--Reynolds states whose differential realizes explicit tangent
	directions and relates Reynolds stress to infinitesimal phase splitting.
	Under a genericity assumption, every deviatoric stress tensor is
	realized by such a first-order mixture defect, yielding
	phase-counting bounds and a minimality result for the observable
	hierarchy.
\end{abstract}

\tableofcontents

\section{Introduction}
\label{sec:introduction}

The incompressible Euler equations occupy a central position in fluid
mechanics and the modern theory of convex integration. One of the key
insights underlying convex integration is that exact solutions are most
effectively studied through the larger class of
\emph{Euler--Reynolds subsolutions}. These are triples
$(v,p,R)$ satisfying
\begin{equation}
	\label{eq:intro-ER}
	\partial_t v+\operatorname{div}(v\otimes v)+\nabla p
	=-\operatorname{div}R,
	\qquad
	\operatorname{div}v=0,
\end{equation}
where the symmetric tensor $R$ is the Reynolds stress. The tensor
$R$ measures the failure of $(v,p)$ to satisfy the Euler equations
exactly and serves as the defect corrected during successive convex
integration iterations
\cite{DeLellisSzekelyhidi2009,DeLellisSzekelyhidi2010,
	DeLellisSzekelyhidi2013,BuckmasterDeLellisSzekelyhidi2016,
	Isett2018,DaneriSzekelyhidi2017}.

A convex integration scheme generates a sequence of Euler--Reynolds
subsolutions whose Reynolds stresses converge to zero while the velocity
fields develop increasingly fine oscillations. Although one typically
has
$v_\alpha\rightharpoonup v,
R_\alpha\to0,$
the nonlinear fluxes satisfy
$v_\alpha\otimes v_\alpha
\not\rightharpoonup
v\otimes v.$
The resulting oscillation defect is the mechanism behind both
nonuniqueness and anomalous dissipation, and naturally leads to the
measure-valued descriptions of weak limits developed by Young, Tartar,
Ball, and many others
\cite{Young1969,Tartar1979,Tartar1983,Ball1989,
	KinderlehrerPedregal1991,Pedregal1997,
	Wiedemann2011,Choffrut2013}.

The present work is motivated by the geometry of these spaces of weak
and approximate solutions. Unlike classical configuration spaces,
Euler--Reynolds subsolutions do not form a smooth manifold. Their
defining constraints are distributional, weak limits need not satisfy
the nonlinear constitutive relation, and there is therefore no canonical
notion of tangent vector from which one may derive a first-order theory.
Consequently, even the basic question of what it means to perturb an
Euler--Reynolds state infinitesimally is far from obvious.

Our starting point is to separate the linear and nonlinear structures of
the Euler equations. Instead of treating the quadratic flux
$v\otimes v$ as part of the differential equation, we introduce it as an
independent variable and consider lifted triples $(v,Q,R)$ satisfying
\begin{equation}
	\label{eq:intro-lifted}
	\partial_t v+\operatorname{div}Q+\nabla p
	=-\operatorname{div}R,
	\qquad
	\operatorname{div}v=0.
\end{equation}
The differential constraint is then completely linear, while all
nonlinearity is concentrated in the algebraic relation
$Q=v\otimes v,$
which defines the Euler locus (Definition~\ref{def:euler-locus}). Weak limits naturally leave this
nonlinear locus while remaining within the ambient linear space,
making the lifted formulation particularly well suited for studying
Euler--Reynolds limits arising from convex integration.

The principal contribution of this paper is to endow lifted
Euler--Reynolds limit spaces with an intrinsic first-order geometry.
To realize this program we work in the category of diffeological spaces
introduced by Souriau and developed systematically by
Iglesias-Zemmour
\cite{JMS,Iglesias2013}. Diffeological spaces generalize smooth
manifolds by replacing coordinate charts with families of compatible
smooth parametrizations, called plots. Every subset inherits a natural
subspace diffeology, smooth maps are defined intrinsically, and internal
tangent spaces exist without requiring any manifold structure \cite{ChristensenWu2016}.
Recent developments have shown that diffeological techniques are
particularly effective for infinite-dimensional spaces and spaces
defined by weak analytical constraints
\cite{Magnot2020,AM2023,ADM2026}.
Therefore, diffeology provides a natural setting for spaces defined
through weak convergence, distributional constraints, or nonlinear
compatibility conditions, where classical manifold techniques are often
inapplicable.

The results of this paper develop the preceding geometric program in
five stages. We first establish the tangent theory for Euler solution
spaces, then extend it to lifted Euler--Reynolds limit spaces, identify
their observable and hidden infinitesimal directions, and finally
develop a finite-dimensional realization through smooth mixtures.

Our first objective is to identify the infinitesimal geometry of spaces
of Euler solutions. We begin with smooth solutions and prove that every
internal tangent vector satisfies the classical linearized Euler
equations (Proposition~\ref{prop:lin-euler-smooth}). We then extend this to weak solutions,
showing that internal tangent directions satisfy the linearized system in the sense of
distributions (Theorem~\ref{thm:weak-distributional-linearization}). Finally, we consider energy-constrained strata and
identify the tangent directions compatible simultaneously with the
linearized dynamics and the prescribed kinetic-energy constraint (Theorem~\ref{thm:admissible}). These
results demonstrate that internal tangent spaces recover the expected
first-order Euler dynamics.

We next investigate the geometry of Euler--Reynolds subsolutions. For
smooth subsolutions with prescribed initial data, we prove that the ambient realization of the internal
tangent space coincides with the full solution space of the linearized
Euler--Reynolds system
(Proposition~\ref{prop:linearized-ER-vpr}).
This stratum possesses an infinite-dimensional family of tangent
directions that modify the Reynolds stress while leaving the velocity
and pressure unchanged
(Corollary~\ref{cor:smooth-stress-gauge}),
revealing a substantial geometric flexibility absent from the exact
Euler equations. In contrast, the zero-stress Euler locus is rigid
(Remark~\ref{rem:euler-locus-rigidity}), motivating the study of weak closures within the lifted
framework.

The central geometric construction of the paper is the lifted limit
space
$\mathbf X(v_0,\mathcal A)$,
defined as the weak closure of smooth strict subsolutions inside the
ambient convenient vector space
$\mathbf E$.
The convenient vector space structure naturally identifies the internal
tangent space of $\mathbf E$ with $\mathbf E$ itself
(see \cite[Theorem~4.5]{ADM2026}), while the inherited subspace
diffeology endows
$\mathbf X(v_0,\mathcal A)$
with intrinsic internal tangent spaces. The inclusion into the ambient
space induces an ambient realization map that associates concrete
distributional perturbations with intrinsic tangent vectors. We prove
that every such realization satisfies the linearized
Euler--Reynolds equations, thereby establishing an intrinsic
distributional linearization for weak lifted limit spaces
(Propositions~\ref{prop:weak-lifted-er}
and~\ref{prop:lin-lifted-er}).
We also describe those tangent directions that remain on the Euler
locus (Proposition~\ref{prop:tangent-euler-locus}).
This distinguishes infinitesimal deformations preserving exact Euler
states from those generating Reynolds stress.

A second theme of the paper concerns observability.
Different physical measurements detect different classes of
infinitesimal perturbations, leading naturally to a hierarchy of
observable tangent directions.
For the velocity, velocity--flux, and full lifted-state observables, we
characterize the corresponding tangent kernels and identify the hidden
directions that remain invisible to each level of observation (Theorem~\ref{thm:kernel-characterization}).
Among these are the stress-gauge directions, whose velocity and flux are
undetectable while their Reynolds-stress component represents a genuine
first-order deformation (Theorem~\ref{thm:stress-gauge-characterization}). 
Their stress perturbations are absorbed into the pressure through
distributional potentials, and explicit smooth representatives are
constructed 
(Proposition~\ref{prop:explicit-stress-gauge-tensors}).
This observable hierarchy provides a geometric
interpretation of which infinitesimal variations are physically visible
and which remain concealed within the lifted formulation.

The final part of the paper develops a finite-dimensional model that
illustrates the abstract tangent theory.
Finite mixtures of lifted Euler--Reynolds states define smooth maps in
the lifted limit space whose differentials explicitly realize internal
tangent directions.
Within this framework, Reynolds stress appears as the covariance defect
generated by infinitesimal phase splitting.
Under a genericity assumption, every deviatoric stress tensor is
obtained in this manner, yielding phase-counting bounds together
with a minimality result for the observable hierarchy (Theorems~\ref{thm:deviatoric-second-moment-realization}, \ref{thm:minimal-observables}).

\paragraph{\textbf{Related works.}}

The Euler--Reynolds formulation has become the fundamental framework for
modern convex integration in incompressible fluid dynamics. Beginning
with the pioneering work of
De Lellis and Sz\'ekelyhidi
\cite{DeLellisSzekelyhidi2009,DeLellisSzekelyhidi2010,DeLellisSzekelyhidi2013},
Euler--Reynolds subsolutions provide the starting point for iterative
constructions of weak Euler solutions, culminating in the resolution of
Onsager's conjecture
\cite{BuckmasterDeLellisSzekelyhidi2016,Isett2018,DaneriSzekelyhidi2017}.
These developments reveal the central role of Reynolds stress as the
carrier of oscillation defects generated by convex integration.

Weak limits of oscillatory sequences are classically described by
Young measures and related measure-valued frameworks
\cite{Young1969,Tartar1979,Tartar1983,Ball1989,
	KinderlehrerPedregal1991,Pedregal1997}.
For the Euler equations, measure-valued solutions and
Euler--Reynolds subsolutions provide complementary descriptions of
oscillation and concentration phenomena
\cite{Wiedemann2011,Choffrut2013}.
Our approach is compatible with this perspective but shifts the emphasis
from the analysis of weak limits to the differential geometry of the
spaces in which those limits reside.

Unlike previous studies of Euler--Reynolds subsolutions, which focus
primarily on existence, relaxation, or convex integration schemes, this
paper studies the geometry of the solution space itself.
The lifted formulation together with its internal tangent spaces yields
an intrinsic first-order description of weak Euler--Reynolds limits,
providing a unified framework for distributional linearization,
observable tangent structures, and finite-dimensional mixture
realizations.

\paragraph{\textbf{Organization.}}
Section~\ref{sec:diffeology-primer} reviews the necessary background on
diffeological spaces, internal tangent spaces, and the canonical
diffeology of the ambient function spaces used throughout the paper.
Section~\ref{sec:tangent-strata} develops the tangent theory for
Euler solution spaces. We first study smooth solutions, then establish
distributional linearization for weak solutions, and finally analyze
energy-constrained strata.
Section~\ref{sec:lifted-limit-space} introduces the lifted
Euler--Reynolds framework, studies the geometry of the Euler locus,
constructs the lifted limit space
$\mathbf X(v_0,\mathcal A)$, and characterizes its observable kernels
and stress-gauge directions.
Finally,
Section~\ref{sec:mixtures}
develops the finite-mixture model, proves the realization theorem for
deviatoric stresses, and derives phase-counting and minimality
results.

\section{Diffeological preliminaries} \label{sec:diffeology-primer}
	We recall the diffeological notions used throughout the paper: subspace and product diffeologies, internal tangent spaces, and the diffeologies on locally convex function spaces. Standard references are \cite{Iglesias2013,ChristensenWu2016,HM-V,KM1997,ADM2026}.

	\subsection{Basic notions}
		A \emph{domain} is any open subset of a Euclidean space $\mathbb{R}^n$, for some non-negative integer $n$, endowed with its standard topology. Any map from a domain to a set $X$ is called a \emph{parametrization} in $X$. A parametrization $ P $ in $ X $ with $ 0\in\mathrm{dom}(P) $ and $ P(0)=x $ is called a parametrization \emph{centered} at $ x $.
	
	\begin{definition}
		A \emph{diffeological space} $(X,\mathfrak{D})$ consists of a set $X$ together with a diffeology $\mathfrak{D}$, which is a collection of parametrizations in $X$, called \emph{plots}, satisfying the following axioms:
		\begin{enumerate}
			\item[\textbf{D1.}]
			The union of the images of all plots is $X$.
			
			\item[\textbf{D2.}]
			If $P:U\to X$ is a plot and $F:V\to U$ is a smooth map between domains, then $P\circ F$ is also a plot.
			
			\item[\textbf{D3.}]
			If $P:U\to X$ is a parametrization such that for every point $r\in U$ there exists an open neighborhood $V\subseteq U$ of $r$ with $P|_V$ a plot, then $P$ is a plot.
		\end{enumerate}
		When the diffeology is clear from the context, we simply write $X$ instead of $(X,\mathfrak{D})$.
	\end{definition}
	
	\begin{example}
		Any smooth manifold carries a natural diffeology whose plots are usual smooth parametrizations. In particular, every domain is naturally a diffeological space.
	\end{example}
	
	\begin{example}
		Let $X$ be a topological space.
		The collection of all continuous parametrizations in $X$ defines a diffeology on $X$, called the \emph{continuous diffeology}.
	\end{example}
	
	\begin{example}
		Let $X$ be any set. The collection of all locally constant parametrizations in $X$ forms a diffeology on $X$, called the \emph{discrete diffeology}. The collection of all parametrizations in $X$ also constitutes a diffeology on $X$, called the \emph{indiscrete diffeology}.
	\end{example}
 
	\begin{definition}\label{def:smooth-maps}
		Let $X$ and $Y$ be diffeological spaces. A map $f:X\to Y$ is said to be \emph{smooth} if for every plot $P$ in $X$, the composition $f\circ P$ is a plot in $Y$.  
	\end{definition}
	
	Indeed, the smooth parametrizations into a diffeological space in the sense of Definition \ref{def:smooth-maps} are precisely the plots. In particular, smooth curves and $1$-plots are the same.
	
	\begin{notation}
		For diffeological spaces $X$ and $Y$, we write $C^\infty(X,Y)$ for the set of all smooth maps from $X$ to $Y$. In particular, we write $C^{\infty}(X)$ for $C^{\infty}(X,\mathbb{R})$.
	\end{notation}
	
	\begin{definition}
		Let $X$ be a diffeological space. A \emph{diffeological subspace} of $X$ is a subset $X'\subseteq X$ endowed with the \emph{subspace diffeology}, whose plots are exactly the plots in $X$ with values in $X'$.
	\end{definition}
	 
	\begin{definition}
		Let $\{X_i\}_{i\in J}$ be a family of diffeological spaces. 
		The \emph{product diffeology} on
		$ X=\prod_{i\in J} X_i$ is given by the parametrizations $ P $ in $ X $ for which
		$ \pi_i\circ P$ is a plot in $X_i$ for all $ i\in J $, where
		$ \pi_i:X\rightarrow X_i$ is the natural projection.
		In the case where $J=\{1,\dots,n\}$ is finite, a parametrization
		$P:U\to X_1\times\cdots\times X_n$
		is a plot if and only if it is of the form
		$P=(P_1,\dots,P_n),$
		where each $P_i:U\to X_i$ is a plot.
	\end{definition}

	Every diffeological space $X$ carries a natural topology, called the D-\emph{topology}, in which a subset of $X$ is D-\emph{open} if and only if its preimage under every plot is open.
	By \cite[\S 2.9]{Iglesias2013}, every smooth map is continuous with respect to the D-topology.
	
	\subsection{Internal tangent spaces}
	For diffeological spaces, several notions of tangent space appear in the
	literature; see \cite{ChristensenWu2016} for further details. We recall
	the construction of internal tangent spaces, following \cite{ChristensenWu2016,HM-V}.

	Let $X$ be a diffeological space and $x\in X$. The \emph{category of germs of plots centered at $x$}, denoted $\mathcal{G}\mathsf{Plots}_x(X)$, has as objects all plots centered at $x$, and a morphism $Q\stackrel{\mathcal{G}_x(F)}{\longrightarrow} P$ between two such plots $P:U\rightarrow X$ and $Q:V\rightarrow X$ centered at $x$ is the germ class of a smooth map $F:W\rightarrow U$, defined on an open neighborhood $W\subseteq V$ of $0$, satisfying $F(0)=0$ and $Q|_W=P\circ F$. Two such maps define the same germ if they coincide on some open neighborhood of $0$ in $V$.
	
	\begin{definition}\label{def:internal-tangent-space}
		The \emph{internal tangent space} $T_xX$ to $X$ at $x$ is the colimit of the functor
		$T_0:\mathcal{G}\mathsf{Plots}_x(X)\rightarrow \mathsf{Vect}$
		from $\mathcal{G}\mathsf{Plots}_x(X)$ to the category of vector spaces and linear maps,
		given by
		\begin{center}
			$Q\stackrel{\mathcal{G}_x(F)}{\longrightarrow} P\qquad\longmapsto\qquad
			dF_0:T_0V\longrightarrow T_0U$,
		\end{center}
		where $dF_0$ is the usual differential of $F$ at $0$.
		The colimit construction provides the following cocone:
		$$
		\xymatrix{
			& T_x X  & \\
			T_0 V  \ar[ur]^{dQ_0} \ar[rr]_{dF_0} & & T_0U  \ar[ul]_{dP_0}
		}
		$$
		where $dP_0:T_0U\rightarrow T_xX$ denotes the canonical linear map into the colimit corresponding to the plot $P$, and similarly for $dQ_0:T_0V\rightarrow T_xX$.
	\end{definition}
	 
	\begin{notation}\label{not:curve-velocity}
		For a smooth curve $\gamma$ in $X$, $\left.\frac{d}{dt}\right|_{t=0}\gamma$ denotes the element $d\gamma_0\!\left(\tfrac{d}{dt}\right)$, where $\tfrac{d}{dt}$ is the standard basis vector of $T_0\mathbb{R}\cong\mathbb{R}$, as in \cite{ChristensenWu2016}.
	\end{notation}
	\begin{remark}\label{rem:curve-representation-finite-sums}
	By
	\cite[Proposition~3.3]{ChristensenWu2016}, every internal tangent vector of $T_xX$ is
	a finite linear combination of representable tangent vectors, i.e. tangent
	vectors of the form
	$\left.\frac{d}{dt}\right|_{t=0}\gamma$
	for $1$-plots 
	$\gamma$ in $X$ centered at $x$.
	\end{remark}
	\begin{definition}\label{def:tangent-map}
		Let $f:X\rightarrow Y$ be a smooth map and $x\in X$. The universal property of colimits induces a unique linear map
		$df_x:T_xX\rightarrow T_{f(x)}Y$, called the \emph{internal tangent map} of $f$ at $x$, characterized by
		$df_x\circ dP_0= d(f\circ P)_0$
		for all plots $P$ in $X$ centered at $x$.
	\end{definition}
	
	This construction yields a functor from the category of pointed diffeological spaces to the category of vector spaces (see \cite[p.~11]{ChristensenWu2016}). In particular, one obtains the standard chain rule for internal tangent spaces.
	
	\begin{proposition}\label{prop:chain-rule}
		Let $f:X\rightarrow Y$ and $g:Y\rightarrow Z$ be smooth maps, and $x\in X$.
		\begin{enumerate}
			\item[(1)] $d(\mathrm{id}_X)_x=\mathrm{id}_{T_xX}$.
			\item[(2)] $d(g\circ f)_x=dg_{f(x)}\circ df_x$.
			\item[(3)] If $f$ is a diffeomorphism, then $df_x$ is an isomorphism and $(df_x)^{-1}=d(f^{-1})_{f(x)}$.
		\end{enumerate}
	\end{proposition}
	\begin{proposition}[{\cite[Proposition 2.42]{ARA}}]\label{prop:const-vanish}
		If $ f:X\rightarrow Y  $ is locally constant and $ x\in X $, then
		$ df_x=0 $.
	\end{proposition}	

\subsection{Canonical diffeology on ambient spaces}
Throughout, LCTVS abbreviates
Hausdorff locally convex topological vector space, and a LCTVS $E$ is
called \emph{convenient} if it is $c^\infty$-complete, i.e.\ every
Mackey--Cauchy sequence in $E$ converges; equivalently, every smooth curve
$\mathbb R\to E$ admits a smooth antiderivative \cite[Theorem~2.14]{KM1997}.

Our ambient spaces are LCTVSs rather than
finite-dimensional manifolds. We therefore recall their natural diffeological
structure and the role of convenience in identifying internal tangent spaces
with the ambient vector space itself. 

\begin{definition}\label{def:canonical-diffeology}
	Let $E$ be a LCTVS.
	\begin{itemize}
		\item The \emph{canonical diffeology} on $E$ consists of all
		parametrizations $P:U\to E$ such that $\ell\circ P:U\to\mathbb R$ is
		smooth for every continuous linear functional $\ell:E\to\mathbb R$.
		Such a $P$ is called \emph{scalarwise smooth}, or \emph{weakly smooth}.
		
		\item The \emph{$c^\infty$-diffeology} on $E$ consists of all
		parametrizations $P:U\to E$ such that, for every $C^\infty$ curve
		$c:\mathbb R\to U$, the composite $P\circ c:\mathbb R\to E$ is a
		$C^\infty$ curve in the sense of \cite{KM1997}.
	\end{itemize}
\end{definition}

By \cite[Theorem~2.14(4)]{KM1997}, a LCTVS $E$ is convenient if and only if
every scalarwise smooth curve in $E$ is $C^\infty$. This has the following
diffeological reformulation.

\begin{proposition}[{\cite[Proposition~4.2]{ADM2026}}]\label{prop:convenient-diffeology}
	A LCTVS $E$ is convenient if and only if its $c^\infty$-diffeology and its
	canonical diffeology coincide.
\end{proposition}

We use repeatedly the standard fact that a multilinear map between convenient
vector spaces, equipped with their canonical diffeologies, is smooth if and only
if it is bounded; in particular every continuous multilinear map is diffeologically smooth
(see \cite{KM1997,ADM2026}). We also use that a finite product of convenient vector
spaces is convenient and that its canonical diffeology coincides with the
product diffeology \cite[Proposition~3.11]{ADM2026}.

For a LCTVS $E$ and $x\in E$, define the canonical linear map
$$
\Phi_x:E\to T_xE,
\qquad
\Phi_x(v):=\left.\frac{d}{dt}\right|_{t=0}\gamma_v,
\qquad
\gamma_v(t):=x+tv,
$$
so that $\Phi_x(v)$ is the internal tangent vector represented by the affine
line through $x$ with velocity $v$.

\begin{theorem}[{\cite[Theorem~4.5]{ADM2026}}]\label{thm:convenient-tangent}
	Let $E$ be a LCTVS equipped with the canonical diffeology and let $x\in E$.
	Then $E$ is convenient if and only if $\Phi_x:E\to T_xE$ is an isomorphism of
	vector spaces.
\end{theorem}

Throughout the paper, whenever an ambient locally convex vector space is
convenient we use Theorem \ref{thm:convenient-tangent} to identify its internal
tangent spaces canonically with the underlying vector space. The next lemma
records that under this identification the internal velocity of a curve is its
Mackey derivative; it is what allows tangent computations to be carried out by
ordinary calculus.

\begin{lemma}\label{lem:velocity}
	Let $F$ be a convenient vector space with its canonical diffeology and let
	$c:\mathbb R\to F$ be a $C^\infty$ curve. Then
	$$
	\left.\frac{d}{dt}\right|_{t=0}c=\Phi_{c(0)}\bigl(c'(0)\bigr),
	$$
	where $c'(0)$ is the Mackey derivative of $c$ at $0$.
\end{lemma}

\begin{proof}
	Write $x:=c(0)$. By Proposition \ref{prop:convenient-diffeology} the curve $c$
	is a $1$-plot centered at $x$, all Mackey derivatives $c^{(k)}(0)$ exist in $F$,
	and $\ell\circ c$ is smooth with $(\ell\circ c)^{(k)}=\ell\circ c^{(k)}$ for
	every continuous linear functional $\ell$ on $F$.
	
	Define $R:\mathbb R\to F$ by $R(t):=t^{-2}\bigl(c(t)-x-tc'(0)\bigr)$ for
	$t\neq0$ and $R(0):=\tfrac12 c''(0)$. For each continuous linear functional
	$\ell$, Taylor's formula with integral remainder applied to
	$g:=\ell\circ c\in C^\infty(\mathbb R,\mathbb R)$ gives
	$(\ell\circ R)(t)=\int_0^1(1-s)g''(st)\,ds$ for all $t$, which is smooth. Hence
	$R$ is a plot and
	\begin{equation}\label{eq:hadamard-vector}
		c(t)=x+t\,c'(0)+t^2R(t),\qquad t\in\mathbb R .
	\end{equation}
	
	Set $h(u,v):=x+u\,c'(0)+uv\,R(u)$, which is a plot centered
	at $x$. With $\delta(t):=(t,t)$, $\iota(t):=(t,0)$ and $j(t):=(0,t)$ we have,
	by \eqref{eq:hadamard-vector},
	$$
	h\circ\delta=c,\qquad h\circ\iota=\gamma_{c'(0)},\qquad h\circ j\equiv x .
	$$
	Let $e_1,e_2$ be the standard basis of $T_0\mathbb R^2\cong\mathbb R^2$ and
	$\tfrac{d}{dt}$ denotes the standard
	basis vector of $ T_0\mathbb R\cong\mathbb R$. Since $h\circ j$ is constant,
	Proposition \ref{prop:const-vanish} and the chain rule
	(Proposition \ref{prop:chain-rule}) give
	$dh_0(e_2)=d(h\circ j)_0(\tfrac{d}{dt})=0$. As $dh_0$ is linear and
	$d\delta_0(\tfrac{d}{dt})=e_1+e_2$, $d\iota_0(\tfrac{d}{dt})=e_1$, we obtain
	$$
	\left.\frac{d}{dt}\right|_{t=0}c
	=dh_0(e_1+e_2)=dh_0(e_1)
	=\left.\frac{d}{dt}\right|_{t=0}\gamma_{c'(0)}
	=\Phi_x\bigl(c'(0)\bigr). \qedhere
	$$
\end{proof}

\begin{proposition}\label{prop:tangent-linear}
	Let $E_1,\dots,E_m$ and $F$ be convenient vector spaces equipped with their
	canonical diffeologies, and let
	$$
	L:E:=E_1\times\cdots\times E_m\longrightarrow F
	$$
	be a smooth multilinear map. Then for each $x=(x_1,\dots,x_m)\in E$ the
	internal tangent map $dL_x:T_xE\to T_{L(x)}F$ is identified with the linear map
	$$
	(v_1,\dots,v_m)\longmapsto
	\sum_{j=1}^m L(x_1,\dots,x_{j-1},v_j,x_{j+1},\dots,x_m)
	$$
	under the canonical isomorphisms $T_xE\simeq E$ and $T_{L(x)}F\simeq F$ of
	Theorem \ref{thm:convenient-tangent}. In particular, for $m=1$ the internal
	tangent map of a smooth linear map $L:E\to F$ is identified with $L$ itself.
\end{proposition}

\begin{proof}
	By \cite[Proposition~3.11]{ADM2026} the product $E$ is a convenient vector
	space whose canonical diffeology is the product diffeology, so
	Theorem \ref{thm:convenient-tangent} applies to $E$ and to $F$ and yields the
	stated isomorphisms $\Phi_x:E\to T_xE$ and $\Phi_{L(x)}:F\to T_{L(x)}F$. Since
	$\Phi_x$ is surjective, it suffices to compute $dL_x\bigl(\Phi_x(v)\bigr)$ for
	an arbitrary $v=(v_1,\dots,v_m)\in E$.
	
	Let $\gamma_v(t):=(x_1+tv_1,\dots,x_m+tv_m)$ and $c:=L\circ\gamma_v$. By the
	chain rule (Proposition \ref{prop:chain-rule}),
	\begin{equation}\label{eq:chain-on-gamma}
		dL_x\bigl(\Phi_x(v)\bigr)
		=dL_x\!\left(\left.\frac{d}{dt}\right|_{t=0}\gamma_v\right)
		=\left.\frac{d}{dt}\right|_{t=0}c .
	\end{equation}
	Multilinearity expands $c$ as the finite sum
	$$
	c(t)=\sum_{I\subseteq\{1,\dots,m\}}t^{|I|}a_I,
	\qquad
	a_I:=L(y_1^I,\dots,y_m^I),
	\qquad
	y_j^I:=
	\begin{cases}
		v_j,& j\in I,\\
		x_j,& j\notin I .
	\end{cases}
	$$
	Thus $c$ is a polynomial curve with coefficients in $F$; in particular it is a
	$C^\infty$ curve and its Mackey derivative at $0$ is the coefficient of $t$,
	namely
	$$
	c'(0)=\sum_{|I|=1}a_I=\sum_{j=1}^m L(x_1,\dots,x_{j-1},v_j,x_{j+1},\dots,x_m).
	$$
	Combining \eqref{eq:chain-on-gamma} with Lemma \ref{lem:velocity} and
	$c(0)=L(x)$ gives
	$$
	dL_x\bigl(\Phi_x(v)\bigr)
	=\Phi_{L(x)}\bigl(c'(0)\bigr)
	=\Phi_{L(x)}\!\left(\sum_{j=1}^m L(x_1,\dots,x_{j-1},v_j,x_{j+1},\dots,x_m)\right),
	$$
	which is the stated formula. For $m=1$ the sum reduces to $L(v)$.
\end{proof}

\begin{example}\label{exa:quad}
	Let $V$ be a convenient space of vector fields admitting a continuous linear
	injection $V\hookrightarrow L^\infty\bigl(0,T;L^2(\mathbb T^d;\mathbb R^d)\bigr)$,
	and set
	$
	W:=L^\infty\bigl(0,T;L^1(\mathbb T^d;\mathbb R^{d\times d})\bigr),
	$
	both equipped with their canonical diffeologies; $W$ is a Banach space and
	hence convenient. 
	Consider the bilinear map
	$$
	B:V\times V\to W,\qquad B(v,w):=v\otimes w.
	$$
	Since the linear injection $\iota:V\hookrightarrow L^\infty\bigl(0,T;L^2\bigr)$ is
	continuous, there is a continuous seminorm $p$ on $V$ with
	$\|v\|_{L^\infty_tL^2_x}\le p(v)$ for all $v\in V$. Hence
	$$
	\|B(v,w)\|_{W}=\|v\otimes w\|_{L^\infty_tL^1_x}
	\le\|v\|_{L^\infty_tL^2_x}\|w\|_{L^\infty_tL^2_x}\le p(v)\,p(w),
	$$
	so $B$ is a continuous bilinear map, therefore bounded, and hence smooth.
	By Proposition \ref{prop:tangent-linear},
	$$
	dB_{(v,w)}(h,k)=h\otimes w+v\otimes k .
	$$
	The diagonal $\Delta:V\to V\times V$, $\Delta(v):=(v,v)$, is linear and
	continuous, hence smooth, so the quadratic map
	$$
	qu:=B\circ\Delta:V\to W,\qquad qu(v)=v\otimes v,
	$$
	is smooth, and by the chain rule (Proposition \ref{prop:chain-rule}),
	$$
	d(qu)_v(h)=dB_{(v,v)}(h,h)=h\otimes v+v\otimes h .
	$$
\end{example}

\section{Internal tangent directions to Euler solution strata}
\label{sec:tangent-strata}

We develop a description of the internal tangent space to
the solution strata of the incompressible Euler equations.
The analysis proceeds in three stages of increasing generality.
We linearize the Euler operator along smooth solutions,
extend the linearization distributionally to weak solutions, and incorporate energy constraints as infinitesimal
admissibility conditions on tangent directions.

\subsection{Smooth solutions and the linearized Euler operator}
\label{sec:smooth-euler}

We first analyze the internal tangent geometry of the space of smooth
solutions of the incompressible Euler equations. On the torus
$\mathbb T^d$, these equations are
\begin{equation}
	\label{eq:euler}
	\partial_t v+\operatorname{div}(v\otimes v)+\nabla p=0,
	\qquad
	\operatorname{div}v=0.
\end{equation}
Here and below, the divergence of a matrix-valued field is taken row-wise, and
$(\nabla\psi)_{ij}=\partial_j\psi_i$, so that
$\langle\operatorname{div}M,\psi\rangle=-\int M:\nabla\psi$ and
$M^{T}\!:\!A=M:A^{T}$.

Set
$\mathcal V
:=
C^\infty([0,T]\times\mathbb T^d;\mathbb R^d),$
which is a Fréchet space and, in
particular, a convenient vector space. Hence, by
Theorem~\ref{thm:convenient-tangent}, for every $v\in\mathcal V$ there is a
canonical vector-space isomorphism
$\Phi_v^{\mathcal V}:\mathcal V\xrightarrow{\;\cong\;}T_v\mathcal V.$
Define the space of smooth Euler solutions by
$$
\mathcal S^\infty
:=
\bigl\{
v\in\mathcal V \;\bigm|\;
\text{$(v,p)$ satisfies \eqref{eq:euler} for some }
p\in C^\infty([0,T]\times\mathbb T^d)
\bigr\}.
$$
We equip $\mathcal S^\infty$ with the subspace diffeology induced by
$\mathcal V$ and denote the inclusion by
$\iota:\mathcal S^\infty\hookrightarrow\mathcal V.$
For a prescribed divergence-free initial datum
$v_0\in C^\infty(\mathbb T^d;\mathbb R^d)$, define
$$
\mathcal S^\infty_{v_0}
:=
\bigl\{
v\in\mathcal S^\infty\mid v(0,\cdot)=v_0
\bigr\},
$$
again with the induced subspace diffeology.
 
\begin{proposition}
	\label{prop:lin-euler-smooth}
	Let $\iota:\mathcal S^\infty\hookrightarrow\mathcal V$ denote the inclusion,
	let $v\in\mathcal S^\infty$, and let $\xi\in T_v\mathcal S^\infty$. Define the
	ambient realization of $\xi$ by
	$ 
	w:=\bigl(\Phi_v^{\mathcal V}\bigr)^{-1}\bigl(d\iota_v(\xi)\bigr)\in\mathcal V .
	$ 
	Then $\operatorname{div}w=0$ on $(0,T)\times\mathbb T^d$, and there is a unique
	$q\in C^\infty([0,T]\times\mathbb T^d;\mathbb R)$ with vanishing spatial mean
	such that
	\begin{equation}
		\label{eq:lin-euler}
		\partial_t w+\operatorname{div}\bigl(v\otimes w+w\otimes v\bigr)+\nabla q=0
		\qquad\text{on }(0,T)\times\mathbb T^d .
	\end{equation}
	If moreover $v\in\mathcal S^\infty_{v_0}$ and $\xi\in T_v\mathcal S^\infty_{v_0}$,
	and if the ambient realization is taken along the composite inclusion
	$$
	\mathcal S^\infty_{v_0}
	\xrightarrow{\ \jmath\ }\mathcal S^\infty
	\xrightarrow{\ \iota\ }\mathcal V,
	\qquad
	\iota_{v_0}:=\iota\circ\jmath,
	\qquad
	w:=\bigl(\Phi_v^{\mathcal V}\bigr)^{-1}\bigl(d(\iota_{v_0})_v(\xi)\bigr),
	$$
	then \eqref{eq:lin-euler} holds and in addition
	$$
	w(0,\cdot)=0.
	$$
\end{proposition}

\begin{proof}
	Since $\mathbb T^d$ is closed, $\operatorname{div}\operatorname{div}(u_1\otimes u_2)(t,\cdot)$
	has vanishing spatial mean for every $t$, so $(-\Delta)^{-1}_0$ may be applied.
	On the Fourier side $(-\Delta)^{-1}_0$ is the multiplier $|k|^{-2}\mathbf 1_{k\neq0}$,
	which acts fibrewise in $t$, commutes with $\partial_t$, and satisfies
	$\|(-\Delta)^{-1}_0f\|_{H^{s}_x}\le 2\,\|f\|_{H^{s-2}_x}$ for zero-mean $f$.
	Since $C^\infty([0,T]\times\mathbb T^d)=\bigcap_{k,s}C^k\bigl([0,T];H^s(\mathbb T^d)\bigr)$
	with the corresponding seminorms, $(-\Delta)^{-1}_0$ preserves joint smoothness and is
	continuous for the Fr\'echet seminorms.
	Define
	$$
	\mathbf B(u_1,u_2):=\operatorname{div}(u_1\otimes u_2),
	\qquad
	\mathbf p(u_1,u_2):=(-\Delta)^{-1}_0\operatorname{div}\operatorname{div}(u_1\otimes u_2).
	$$
	Pointwise tensor multiplication
	$\mathcal V\times\mathcal V\to C^\infty([0,T]\times\mathbb T^d;\mathbb R^{d\times d})$
	is continuous bilinear, and $\operatorname{div}$,
	$(-\Delta)^{-1}_0\operatorname{div}\operatorname{div}$ are continuous linear on
	the relevant Fr\'echet spaces. Hence $\mathbf B$ and $\mathbf p$ are continuous
	bilinear, so the associated quadratic maps
	$$
	\mathcal P(u):=\mathbf p(u,u),
	\qquad
	\mathcal Q(u):=\mathbf B(u,u)+\nabla\mathbf p(u,u)
	$$
	are smooth for the canonical diffeologies with 
	\begin{equation}
		\label{eq:polarization}
		d\mathcal P_v(w)=\mathbf p(v,w)+\mathbf p(w,v),
		\qquad
		d\mathcal Q_v(w)=\mathbf B(v,w)+\mathbf B(w,v)
		+\nabla\bigl(\mathbf p(v,w)+\mathbf p(w,v)\bigr).
	\end{equation}
	The map $\mathcal P$ is characterized by
	$$
	-\Delta\mathcal P(u)=\operatorname{div}\operatorname{div}(u\otimes u),
	\qquad
	\int_{\mathbb T^d}\mathcal P(u)(t,x)\,dx=0
	\quad\text{for all }t\in[0,T].
	$$
	Let $v\in\mathcal S^\infty$ and let $p$ be a pressure as in the definition of
	$\mathcal S^\infty$, so that
	$$\partial_t v+\operatorname{div}(v\otimes v)+\nabla p=0$$ 
	and
	$\operatorname{div}v=0$. Taking the spatial divergence of the momentum equation
	and using $\operatorname{div}v=0$, hence $\partial_t\operatorname{div}v=0$,
	gives
	$$
	-\Delta p=\operatorname{div}\operatorname{div}(v\otimes v)=-\Delta\mathcal P(v).
	$$
	For each fixed $t$ the function $p(t,\cdot)-\mathcal P(v)(t,\cdot)$ is harmonic
	on the compact connected manifold $\mathbb T^d$, hence constant; so
	$p-\mathcal P(v)$ is a function of $t$ alone and in particular
	$\nabla p=\nabla\mathcal P(v)$. Therefore
	\begin{equation}
		\label{eq:euler-normalized}
		\partial_t v+\operatorname{div}(v\otimes v)+\nabla\mathcal P(v)=0,
		\qquad
		\operatorname{div}v=0
		\qquad\text{for every }v\in\mathcal S^\infty .
	\end{equation}
	Set
	$$
	\mathcal D:\mathcal V\to C^\infty([0,T]\times\mathbb T^d;\mathbb R),
	\qquad
	\mathcal D(u):=\operatorname{div}u,
	$$
	$$
	\mathcal R:\mathcal V\to C^\infty([0,T]\times\mathbb T^d;\mathbb R^d),
	\qquad
	\mathcal R(u):=\partial_t u+\operatorname{div}(u\otimes u)+\nabla\mathcal P(u).
	$$
	The map $\mathcal D$ is continuous linear, hence smooth, and
	$d\mathcal D_v=\mathcal D$ for every $v$. The map $\mathcal R$ is not
	multilinear: it is a polynomial of degree two, and we decompose it accordingly
	as
	$$
	\mathcal R=\mathcal L+\mathcal Q,
	\qquad
	\mathcal L(u):=\partial_t u,
	$$
	with $\mathcal L$ continuous linear and $\mathcal Q$ the quadratic map as above. Hence $\mathcal R$ is smooth and, by \eqref{eq:polarization},
	\begin{equation}
		\label{eq:dR}
		d\mathcal R_v(w)
		=\partial_t w
		+\operatorname{div}(v\otimes w+w\otimes v)
		+\nabla\bigl(\mathbf p(v,w)+\mathbf p(w,v)\bigr).
	\end{equation}
	By \eqref{eq:euler-normalized} both observables vanish along the inclusion:
	$
	\mathcal D\circ\iota=0,
	\mathcal R\circ\iota=0 .
	$
	Since $\mathcal D\circ\iota$ is constant, Proposition \ref{prop:const-vanish} and the chain rule (Proposition \ref{prop:chain-rule}) yield
	$$
	0=d(\mathcal D\circ\iota)_v(\xi)
	=d\mathcal D_v\bigl(d\iota_v(\xi)\bigr)
	=d\mathcal D_v\bigl(\Phi_v^{\mathcal V}(w)\bigr)
	=\operatorname{div}w,
	$$
	where the last equality uses $d\mathcal D_v=\mathcal D$ together with the
	identification $\Phi_v^{\mathcal V}$. In the same way,
	$$
	0=d(\mathcal R\circ\iota)_v(\xi)
	=d\mathcal R_v\bigl(d\iota_v(\xi)\bigr)
	=d\mathcal R_v\bigl(\Phi_v^{\mathcal V}(w)\bigr),
	$$
	so that, by \eqref{eq:dR},
	$$
	\partial_t w+\operatorname{div}(v\otimes w+w\otimes v)
	+\nabla\bigl(\mathbf p(v,w)+\mathbf p(w,v)\bigr)=0 .
	$$
	Setting
	$$
	q:=d\mathcal P_v(w)=\mathbf p(v,w)+\mathbf p(w,v)
	=(-\Delta)^{-1}_0\operatorname{div}\operatorname{div}\bigl(v\otimes w+w\otimes v\bigr)
	$$
	gives \eqref{eq:lin-euler}, and $q\in C^\infty([0,T]\times\mathbb T^d;\mathbb R)$
	with vanishing spatial mean because $(-\Delta)^{-1}_0$ maps into the zero-mean
	subspace and preserves smoothness. Moreover, $q$ is unique: if
	$q_1,q_2$ both satisfy \eqref{eq:lin-euler}, then $\nabla(q_1-q_2)=0$, so
	$q_1-q_2$ depends only on $t$, and the zero-mean normalization forces
	$q_1=q_2$.
	
    Finally, let
	$$
	\operatorname{ev}_0:\mathcal V\to C^\infty(\mathbb T^d;\mathbb R^d),
	\qquad
	\operatorname{ev}_0(u):=u(0,\cdot),
	$$
	which is continuous linear, hence smooth, with
	$d(\operatorname{ev}_0)_v=\operatorname{ev}_0$ for every $v$. By definition
	$$
	\mathcal S^\infty_{v_0}
	=\mathcal S^\infty\cap\operatorname{ev}_0^{-1}(v_0),
	$$
	so the composite $\operatorname{ev}_0\circ\iota_{v_0}$ is the constant map
	$v_0$. Applying the chain rule to
	$\operatorname{ev}_0\circ\iota_{v_0}:\mathcal S^\infty_{v_0}\to
	C^\infty(\mathbb T^d;\mathbb R^d)$ at $v$ therefore gives
	$$
	0=d(\operatorname{ev}_0\circ\iota_{v_0})_v(\xi)
	=d(\operatorname{ev}_0)_v\bigl(d(\iota_{v_0})_v(\xi)\bigr)
	=\operatorname{ev}_0\bigl(\Phi_v^{\mathcal V}(w)\bigr)
	=w(0,\cdot).
	$$
	Since $\iota_{v_0}=\iota\circ\jmath$, the chain rule
	gives
	$d(\iota_{v_0})_v=d\iota_v\circ d\jmath_v$, so the ambient realization $w$
	produced here coincides with the one used above applied to
	$d\jmath_v(\xi)\in T_v\mathcal S^\infty$. Hence \eqref{eq:lin-euler} also holds
	for this $w$, which completes the proof.
\end{proof}

\begin{remark}
	\label{rem:linearized-inclusion}
	Proposition~\ref{prop:lin-euler-smooth} gives the inclusion
	$$
	d\iota_v\bigl(T_v\mathcal S^\infty\bigr)
	\subseteq
	\Phi_v^{\mathcal V}
	\left(
	\left\{
	w\in\mathcal V \;\middle|\;
	\begin{array}{l}
		\operatorname{div}w=0,\ \text{and there exists }q\in
		C^\infty([0,T]\times\mathbb T^d)\\
		\text{such that } \partial_t w
		+\operatorname{div}(v\otimes w+w\otimes v)
		+\nabla q=0
	\end{array}
	\right\}
	\right).
	$$
	For the fixed-initial-data solution space, the right-hand side is further
	restricted by
	$w(0,\cdot)=0.$
	
	No converse statement is asserted. In particular, a solution of the
	linearized Euler equation need not be tangent to a smooth curve of exact
	Euler solutions. Thus the ambient linearized equation provides a necessary
	condition for internal realizability, but not a sufficient one.
\end{remark}

\subsection{Distributional linearization of weak solutions}
\label{sec:distributional-linearization}
The purpose of this subsection is to show that every internal tangent vector to the weak solution
stratum, when realized in the ambient Banach space, satisfies the
distributional linearized Euler system.

Let
$E:=L^\infty\bigl(0,T;L^2(\mathbb T^d;\mathbb R^d)\bigr),$
equipped with the canonical diffeology associated with its Banach topology.
Since $E$ is a Banach space, it is convenient. Hence, by
Theorem~\ref{thm:convenient-tangent}, for every $v\in E$ the canonical map
$\Phi_v^E:E\overset{\cong}{\longrightarrow}T_vE,$
is a vector-space isomorphism.

Let
$v_0\in L^2(\mathbb T^d;\mathbb R^d)$
be a prescribed divergence-free initial datum. We denote by
$\mathcal S^w(v_0)\subseteq E$
the set of weak solutions of the incompressible Euler equations on
$[0,T]\times\mathbb T^d$ with initial datum $v_0$, in the sense of the weak
formulation. Thus $v\in\mathcal S^w(v_0)$ means that $v$ is weakly
divergence-free and satisfies
\begin{equation}
	\label{eq:weak-euler}
	\int_0^T\!\!\int_{\mathbb T^d}
	\Bigl(
	v\cdot \partial_t\varphi
	+
	(v\otimes v):\nabla\varphi
	\Bigr)\,dx\,dt
	+
	\int_{\mathbb T^d} v_0(x)\cdot\varphi(0,x)\,dx
	=0
\end{equation}
for every test field
$$
\varphi\in C_c^\infty([0,T)\times\mathbb T^d;\mathbb R^d),
\qquad
\operatorname{div}\varphi=0.
$$
We equip $\mathcal S^w(v_0)$ with the subspace diffeology induced by $E$,
and denote the inclusion by
$\iota:\mathcal S^w(v_0)\hookrightarrow E.$

We first establish the elementary smoothness properties of the ambient
functionals used below.

\begin{lemma}
	\label{lem:weak-observables}
	Let
	$\psi\in C_c^\infty([0,T)\times\mathbb T^d;\mathbb R^d),$
	$A\in C_c^\infty([0,T)\times\mathbb T^d;\mathbb R^{d\times d}).$
	Define
	$$
	F_\psi:E\to\mathbb R,
	\qquad
	F_\psi(u)
	:=
	\int_0^T\!\!\int_{\mathbb T^d}
	u\cdot\partial_t\psi\,dx\,dt,
	$$
	and
	$$
	G_A:E\to\mathbb R,
	\qquad
	G_A(u)
	:=
	\int_0^T\!\!\int_{\mathbb T^d}
	(u\otimes u):A\,dx\,dt.
	$$
	Then $F_\psi$ and $G_A$ are smooth. Moreover, for every $u,h\in E$,
	\begin{align}
		d(F_\psi)_u(h)
		&=
		\int_0^T\!\!\int_{\mathbb T^d}
		h\cdot\partial_t\psi\,dx\,dt,
		\label{eq:weak-dF}
		\\
		d(G_A)_u(h)
		&=
		\int_0^T\!\!\int_{\mathbb T^d}
		(u\otimes h+h\otimes u):A\,dx\,dt.
		\label{eq:weak-dG}
	\end{align}
\end{lemma}

\begin{proof}
	The functional $F_\psi$ is continuous linear on $E$, hence smooth.
	By Example \ref{exa:quad}, $u\mapsto u\otimes u$ is a smooth quadratic map. Pairing
	against the fixed compactly supported test tensor $A$ is continuous
	linear on $L^\infty_tL^1_x$, and hence $G_A$ is smooth.
	The differential formulas follow from the expansion
	$$
	(u+sh)\otimes(u+sh)
	=
	u\otimes u
	+s(u\otimes h+h\otimes u)
	+s^2(h\otimes h).
	$$
\end{proof}

For each divergence-free test field
$$
\psi\in C_c^\infty([0,T)\times\mathbb T^d;\mathbb R^d),
\qquad
\operatorname{div}\psi=0,
$$
define
$\mathcal J_\psi:E\to\mathbb R$
by
$$
\mathcal J_\psi(u)
:=
\int_0^T\!\!\int_{\mathbb T^d}
\Bigl(
u\cdot\partial_t\psi
+
(u\otimes u):\nabla\psi
\Bigr)\,dx\,dt
+
\int_{\mathbb T^d}v_0(x)\cdot\psi(0,x)\,dx.
$$
By Lemma~\ref{lem:weak-observables}, $\mathcal J_\psi$ is smooth. Moreover,
by the weak Euler identity,
$$
\mathcal J_\psi|_{\mathcal S^w(v_0)}\equiv 0.
$$

We here linearize the incompressibility constraint.

\begin{proposition}
	\label{prop:weak-linearized-divergence}
	Let $v\in\mathcal S^w(v_0)$, let
	$\xi\in T_v\mathcal S^w(v_0)$, and let
	$w:=(\Phi_v^E)^{-1}\bigl(d\iota_v(\xi)\bigr)\in E$
	be its ambient realization. Then
	$\operatorname{div} w=0$
	in $\mathcal D'((0,T)\times\mathbb T^d)$.
\end{proposition}

\begin{proof}
	Let
	$\phi\in C_c^\infty((0,T)\times\mathbb T^d).$
	Define
	$$
	H_\phi:E\to\mathbb R,
	\qquad
	H_\phi(u)
	:=
	\int_0^T\!\!\int_{\mathbb T^d}
	u\cdot\nabla\phi\,dx\,dt.
	$$
	This is a continuous linear functional on $E$, hence smooth. Since every
	element of $\mathcal S^w(v_0)$ is weakly divergence-free,
	$H_\phi|_{\mathcal S^w(v_0)}\equiv 0.$
	Therefore, by the chain rule (Proposition \ref{prop:chain-rule}),
	$$
	0
	=
	d(H_\phi\circ\iota)_v(\xi)
	=
	d(H_\phi)_v\bigl(d\iota_v(\xi)\bigr).
	$$
	Since $d\iota_v(\xi)=\Phi_v^E(w)$, the canonical tangent identification
	gives
	$$
	d(H_\phi)_v\bigl(d\iota_v(\xi)\bigr)
	=
	\int_0^T\!\!\int_{\mathbb T^d}
	w\cdot\nabla\phi\,dx\,dt.
	$$
	Hence
	$$
	\int_0^T\!\!\int_{\mathbb T^d}
	w\cdot\nabla\phi\,dx\,dt=0
	$$
	for every $\phi\in C_c^\infty((0,T)\times\mathbb T^d)$. Therefore
	$\operatorname{div} w=0$
	in $\mathcal D'((0,T)\times\mathbb T^d)$.
\end{proof}

\begin{proposition}
	\label{prop:weak-linearized-divfree}
	Let $v\in\mathcal S^w(v_0)$, let
	$\xi\in T_v\mathcal S^w(v_0)$, and let
	$w:=(\Phi_v^E)^{-1}\bigl(d\iota_v(\xi)\bigr)\in E$
	be its ambient realization. Then for every
	$$
	\psi\in C_c^\infty([0,T)\times\mathbb T^d;\mathbb R^d),
	\qquad
	\operatorname{div}\psi=0,
	$$
	one has
	\begin{equation}
		\label{eq:weak-linearized-divfree}
		\int_0^T\!\!\int_{\mathbb T^d}
		\Bigl(
		w\cdot\partial_t\psi
		+
		(v\otimes w+w\otimes v):\nabla\psi
		\Bigr)\,dx\,dt
		=0.
	\end{equation}
\end{proposition}

\begin{proof}
	For such a test field $\psi$, we have
	$\mathcal J_\psi|_{\mathcal S^w(v_0)}\equiv 0.$
	Hence, by the chain rule (Proposition \ref{prop:chain-rule}),
	$$
	0
	=
	d(\mathcal J_\psi\circ\iota)_v(\xi)
	=
	d(\mathcal J_\psi)_v\bigl(d\iota_v(\xi)\bigr)=
	d(\mathcal J_\psi)_v(\Phi_v^E(w)).
	$$
	The initial-data term in $\mathcal J_\psi$ is constant in $u$, hence its
	differential vanishes. Applying Lemma~\ref{lem:weak-observables} to the
	linear term with test field $\psi$ and to the quadratic term with
	$A=\nabla\psi$, we obtain
	$$
	0
	=
	\int_0^T\!\!\int_{\mathbb T^d}
	\Bigl(
	w\cdot\partial_t\psi
	+
	(v\otimes w+w\otimes v):\nabla\psi
	\Bigr)\,dx\,dt,
	$$
	which is exactly \eqref{eq:weak-linearized-divfree}.
\end{proof}

We use the following distributional Helmholtz--Hodge principle on the
torus.
\begin{lemma}
	\label{lem:distributional-hodge}
	Let
	$F\in \mathcal D'\bigl((0,T)\times\mathbb T^d;\mathbb R^d\bigr)$
	satisfy
	$\langle F,\psi\rangle=0$
	for every
	$\psi\in C_c^\infty((0,T)\times\mathbb T^d;\mathbb R^d),$
	$\operatorname{div}\psi=0.$
	Then there exists
	$\pi\in \mathcal D'\bigl((0,T)\times\mathbb T^d\bigr)$
	such that
	$F=\nabla\pi$
	in 
	$\mathcal D'\bigl((0,T)\times\mathbb T^d;\mathbb R^d\bigr).$
\end{lemma}
\begin{proof}
	For $k\in\mathbb Z^d$ define $\widehat F(k)\in\mathcal D'\bigl((0,T);\mathbb C^d\bigr)$ by
	$$
	\bigl\langle \widehat F_j(k),\zeta\bigr\rangle
	:=\bigl\langle F_j,\ \zeta(t)\,e^{-ik\cdot x}\bigr\rangle,
	\qquad \zeta\in C_c^\infty(0,T),\quad 1\le j\le d .
	$$
	Since $F$ has finite order on a neighbourhood of $\operatorname{supp}\zeta\times\mathbb T^d$,
	there are $C_\zeta,N$ with
	$$|\langle\widehat F_j(k),\zeta\rangle|\le C_\zeta(1+|k|)^{N},$$ so the Fourier
	coefficients grow at most polynomially.
	Fix $k$ and $a\in\mathbb C^d$ with $a\cdot k=0$. The field
	$\psi(t,x)=\zeta(t)\operatorname{Re}\bigl(a\,e^{ik\cdot x}\bigr)$ lies in
	$C_c^\infty((0,T)\times\mathbb T^d;\mathbb R^d)$ and satisfies
	$$\operatorname{div}\psi=\zeta\operatorname{Re}\bigl(i(a\cdot k)e^{ik\cdot x}\bigr)=0;$$
	the same holds with $\operatorname{Im}$ in place of $\operatorname{Re}$. The hypothesis
	therefore gives $a\cdot\widehat F(k)=0$ for every $a\perp k$. Consequently
	$\widehat F(0)=0$, and for $k\neq0$ the vector $\widehat F(k)$ is a multiple of $k$.
	Set $\widehat\pi(0):=0$ and
	$\widehat\pi(k):=-i|k|^{-2}\,k\cdot\widehat F(k)$ for $k\neq0$, so that
	$ik\,\widehat\pi(k)=\widehat F(k)$ for all $k$. Because
	$|\widehat\pi(k)|\lesssim|k|^{-1}|\widehat F(k)|$ still grows polynomially, the series
	$$
	\pi:=\sum_{k\in\mathbb Z^d}\widehat\pi(k)\,e^{ik\cdot x}
	$$
	converges in $\mathcal D'\bigl((0,T)\times\mathbb T^d\bigr)$: for
	$\phi\in C_c^\infty((0,T)\times\mathbb T^d)$ the coefficients
	$\phi_{-k}(t)=\int_{\mathbb T^d}\phi(t,x)e^{ik\cdot x}\,dx$ decay rapidly in every
	$C^m_t$ seminorm. Testing $F$ and $\nabla\pi$ against $\zeta(t)e^{-ik\cdot x}$ for all
	$k$ and $\zeta$ gives $F=\nabla\pi$ in
	$\mathcal D'\bigl((0,T)\times\mathbb T^d;\mathbb R^d\bigr)$, where $\nabla$ is the
	spatial gradient. 
\end{proof}

\begin{theorem}
	\label{thm:weak-distributional-linearization}
	Let
	$v\in\mathcal S^w(v_0),$
	$\xi\in T_v\mathcal S^w(v_0),$
	and let
	$w:=(\Phi_v^E)^{-1}\bigl(d\iota_v(\xi)\bigr)\in E$
	be the ambient realization of $\xi$. Then
	$\operatorname{div} w=0$
	in $\mathcal D'((0,T)\times\mathbb T^d)$, and there exists
	$q\in\mathcal D'((0,T)\times\mathbb T^d)$
	such that
	\begin{equation}
		\label{eq:weak-distributional-linearized-euler}
		\partial_t w
		+
		\operatorname{div}(v\otimes w+w\otimes v)
		+
		\nabla q
		=
		0
	\end{equation}
	holds in
	$\mathcal D'((0,T)\times\mathbb T^d;\mathbb R^d).$
\end{theorem}

\begin{proof}
	The divergence constraint follows from
	Proposition~\ref{prop:weak-linearized-divergence}.

	Define $F\in\mathcal{D}'\bigl((0,T)\times\mathbb{T}^d;\mathbb{R}^d\bigr)$ by
	$$
	\langle F,\psi\rangle
	:= -\int_0^T\!\!\int_{\mathbb{T}^d} w\cdot\partial_t\psi\,\mathrm{d}x\,\mathrm{d}t
	-\int_0^T\!\!\int_{\mathbb{T}^d}
	\bigl(v\otimes w + w\otimes v\bigr):\nabla\psi\,\mathrm{d}x\,\mathrm{d}t,
	\qquad \psi\in C_c^\infty\bigl((0,T)\times\mathbb{T}^d;\mathbb{R}^d\bigr).
	$$
	This is well defined: since $w\in L^\infty_tL^2_x$ and $\psi$ is smooth with compact
	support, the first term is finite; and since $v,w\in L^2_x$, Cauchy--Schwarz gives
	$v\otimes w + w\otimes v\in L^\infty_tL^1_x$, so the second term is finite as well.
	By construction, $F = \partial_t w + \operatorname{div}(v\otimes w + w\otimes v)$ in
	$\mathcal{D}'$.
	Let $\psi\in C_c^\infty\bigl((0,T)\times\mathbb{T}^d;\mathbb{R}^d\bigr)$ be divergence-free,
	$\operatorname{div}\psi = 0$. Since $\operatorname{supp}\psi$ is compact in the open set
	$(0,T)\times\mathbb{T}^d$, we have the inclusion
	$$
	C_c^\infty\bigl((0,T)\times\mathbb{T}^d\bigr)\subseteq
	C_c^\infty\bigl([0,T)\times\mathbb{T}^d\bigr),
	$$
	so $\psi$ is an admissible test field for the weak linearized formulation. Hence
	Proposition~\ref{prop:weak-linearized-divfree} yields $\langle F,\psi\rangle = 0$.
	Thus $F$ annihilates every divergence-free test field. By Lemma~\ref{lem:distributional-hodge}, there exists
	$\pi\in\mathcal{D}'\bigl((0,T)\times\mathbb{T}^d\bigr)$ with $F = \nabla\pi$. Setting
	$q := -\pi$ gives
	$$
	\partial_t w + \operatorname{div}\bigl(v\otimes w + w\otimes v\bigr) + \nabla q
	= 0
	\quad\text{in }\mathcal{D}'\bigl((0,T)\times\mathbb{T}^d;\mathbb{R}^d\bigr),
	$$
	which is the claim.
\end{proof}
 
	Since $E=L^\infty\bigl(0,T;L^2(\mathbb T^d;\mathbb R^d)\bigr)$ carries no canonical trace at $t=0$, the condition
	$w(0,\cdot)=0$ of Proposition~\ref{prop:lin-euler-smooth} has no pointwise
	counterpart here. Nevertheless, \eqref{eq:weak-linearized-divfree} holds for
	every divergence-free
	$\psi\in C_c^\infty([0,T)\times\mathbb T^d;\mathbb R^d)$, in particular for
	$\psi$ with $\psi(0,\cdot)\neq0$, and without any initial term, because
	the datum $v_0$ enters $\mathcal J_\psi$ as an additive constant. This is
	precisely the weak formulation of the linearized Euler system with vanishing
	initial datum. Upgrading it to $w(0,\cdot)=0$ in $L^2(\mathbb T^d;\mathbb R^d)$ requires extra time
	regularity, e.g.\ weak continuity $w\in C_w([0,T];L^2(\mathbb T^d;\mathbb R^d))$.
 
\subsection{Energy constraints and admissible directions}

We impose energy constraints on the diffeological space of weak solutions. Let
$H:=L^2(\mathbb T^d;\mathbb R^d)$ be the instantaneous velocity space and $E:=L^\infty(0,T;H)$ be the ambient
trajectory space. The space $E$ is a Banach space equipped with its
canonical diffeology. Its internal tangent spaces are canonically identified
with $E$ by
Theorem~\ref{thm:convenient-tangent}.
Let $X\subseteq E$ be the diffeological subspace of weak solutions under
consideration. When the initial datum is fixed, one may take
$X=\mathcal S^w(v_0).$

\subsubsection{Energy observables and prescribed profiles}

The instantaneous kinetic-energy observable is
$$
\mathcal E:H\longrightarrow\mathbb R,
\qquad
\mathcal E(u):=\frac12\int_{\mathbb T^d}|u(x)|^2\,dx.
$$
It is smooth, with differential
$$
d\mathcal E_u(h)
=
\int_{\mathbb T^d}u(x)\cdot h(x)\,dx,
\qquad h\in H.
$$
At the trajectory level, define
$\mathcal E_{\mathrm{prof}}:E\longrightarrow L^\infty(0,T)$
by
$$
\mathcal E_{\mathrm{prof}}(v)(t)
:=
\mathcal E(v(t))
=
\frac12\int_{\mathbb T^d}|v(t,x)|^2\,dx
$$
for almost every $t\in(0,T)$. This map is a continuous quadratic map, hence smooth. Indeed,
$\mathcal E_{\mathrm{prof}}(v)=\frac12 B(v,v)$, where
$$
B(v,w)(t):=\int_{\mathbb T^d}v(t,x)\cdot w(t,x)\,dx
$$
defines a continuous bilinear map
$B:E\times E\to L^\infty(0,T).$

If $X$ is restricted to a class of trajectories whose energy profiles are
smooth in time, then the restriction of $\mathcal E_{\mathrm{prof}}$ may
instead be viewed as a map with values in $C^\infty([0,T])$, after
specifying the corresponding diffeology on that target. In the general weak
setting, the natural target is $L^\infty(0,T)$.

Let $e\in L^\infty(0,T)$ be a prescribed energy profile. Define the
energy-constrained solution space by
$$
X_{\mathcal E}(e)
:=
\bigl\{v\in X\mid \mathcal E_{\mathrm{prof}}(v)=e
\text{ in }L^\infty(0,T)\bigr\}.
$$
Equivalently,
\begin{equation}\label{eq:energy-profile}
	\frac12\int_{\mathbb T^d}|v(t,x)|^2\,dx=e(t)
	\qquad\text{for a.e. }t\in(0,T).
\end{equation}
The space $X_{\mathcal E}(e)$ is equipped with the subspace diffeology
inherited from $X$.

For every $\zeta\in C_c^\infty(0,T)$, define
$$
J_\zeta:E\longrightarrow\mathbb R,
\qquad
J_\zeta(v)
:=
\int_0^T\zeta(t)\,\mathcal E(v(t))\,dt.
$$
The map $J_\zeta$ is smooth and its differential is given by
	\begin{equation}\label{eq:dJ}
		d(J_\zeta)_v(w)
		=
		\int_0^T\zeta(t)
		\left(\int_{\mathbb T^d}v(t,x)\cdot w(t,x)\,dx\right)dt .
	\end{equation}
The profile constraint \eqref{eq:energy-profile} is equivalent to the family
of scalar constraints
\begin{equation}\label{eq:scalar-energy-constraint}
	J_\zeta(v)
	=
	\int_0^T\zeta(t)e(t)\,dt
	\qquad
	\text{for every }\zeta\in C_c^\infty(0,T).
\end{equation}
Indeed, the difference
$
\mathcal E_{\mathrm{prof}}(v)-e\in L^\infty(0,T)
$
vanishes if and only if its pairing with every
$\zeta\in C_c^\infty(0,T)$ vanishes.

\subsubsection{Infinitesimal admissibility}

Let $v\in X_{\mathcal E}(e)$ and let
$\xi\in T_vX_{\mathcal E}(e)$
be an arbitrary internal tangent vector. 
Let
$\jmath_e:X_{\mathcal E}(e)\hookrightarrow X$
denote the inclusion and set
$\xi_X:=d(\jmath_e)_v(\xi)\in T_vX.$
Writing
$\iota_X:X\hookrightarrow E$
for the ambient inclusion, define the ambient realization of $\xi$ by
\begin{equation}\label{eq:ambient-realization}
	w
	:=
	(\Phi_v^E)^{-1}
	\bigl(d(\iota_X)_v(\xi_X)\bigr)
	\in E.
\end{equation}

\begin{proposition} 
	\label{prop:orthogonality}
	Let $v\in X_{\mathcal E}(e)$ and
	$\xi\in T_vX_{\mathcal E}(e)$, and let $w\in E$ be its ambient
	realization defined by \eqref{eq:ambient-realization}. Then
	\begin{equation}\label{eq:energy-orthogonality-test}
		d(J_\zeta)_v(w)=0
		\qquad
		\text{for every }\zeta\in C_c^\infty(0,T).
	\end{equation}
	Equivalently,
		$$\int_{\mathbb T^d}v(t,x)\cdot w(t,x)\,dx=0
		\qquad\text{for a.e. }t\in(0,T).$$
\end{proposition}

\begin{proof}
	By \eqref{eq:scalar-energy-constraint}, the restriction of $J_\zeta$ to
	$X_{\mathcal E}(e)$ is constant:
	$$
	J_\zeta|_{X_{\mathcal E}(e)}
	\equiv
	\int_0^T\zeta(t)e(t)\,dt.
	$$
	Hence
	$d\bigl(J_\zeta|_{X_{\mathcal E}(e)}\bigr)_v(\xi)=0.$
	By the chain rule (Proposition \ref{prop:chain-rule}),
	$$
	d\bigl(J_\zeta|_{X_{\mathcal E}(e)}\bigr)_v
	=
	d(J_\zeta)_v
	\circ d(\iota_X)_v
	\circ d(\jmath_e)_v.
	$$
	Using the definition of $w$, this gives
	$d(J_\zeta)_v(w)=0,$
	which proves \eqref{eq:energy-orthogonality-test}. Substituting
	\eqref{eq:dJ}, we obtain
	$$
	\int_0^T\zeta(t)
	\left(
	\int_{\mathbb T^d}v(t,x)\cdot w(t,x)\,dx
	\right)dt=0
	$$
	for every $\zeta\in C_c^\infty(0,T)$. The inner integral belongs to
	$L^\infty(0,T)$, by the Cauchy--Schwarz inequality. Hence, since this pairing vanishes for every
	$\zeta\in C_c^\infty(0,T)$, it follows that
	$$
	\int_{\mathbb T^d}v(t,x)\cdot w(t,x)\,dx=0
	\qquad\text{for a.e. }t\in(0,T).
	$$
\end{proof}

\subsubsection{Compatibility with the linearized Euler equation}

Suppose now that
$X=\mathcal S^w(v_0)$
is the weak Euler solution stratum introduced in
Section~\ref{sec:distributional-linearization}. Thus
$$
X_{\mathcal E}(e)
=
\{v\in\mathcal S^w(v_0)\mid\mathcal E_{\mathrm{prof}}(v)=e\}.
$$

\begin{theorem} 
	\label{thm:admissible}
	Suppose $X=\mathcal S^w(v_0)$. Let
	$v\in X_{\mathcal E}(e)$ and
	$\xi\in T_vX_{\mathcal E}(e)$. Let $w\in E$ be the ambient realization
	of $\xi$ defined by \eqref{eq:ambient-realization}. Then:
	\begin{enumerate}
		\item[(a)] $w$ is distributionally divergence-free:
		\begin{equation}\label{eq:linearized-divergence}
			\operatorname{div}w=0
			\quad\text{in }
			\mathcal D'((0,T)\times\mathbb T^d);
		\end{equation}
		
		\item[(b)] there exists
		$q\in
		\mathcal D'((0,T)\times\mathbb T^d)$
		such that
		\begin{equation}\label{eq:linearized-euler-pressure}
			\partial_tw
			+
			\operatorname{div}
			\bigl(v\otimes w+w\otimes v\bigr)
			+
			\nabla q
			=0
		\end{equation}
		in
		$\mathcal D'((0,T)\times\mathbb T^d;\mathbb R^d)$;
		
		\item[(c)] $w$ satisfies the linearized energy constraint:
		\begin{equation}\label{eq:linearized-energy}
			\int_{\mathbb T^d}v(t,x)\cdot w(t,x)\,dx=0
			\qquad\text{for a.e. }t\in(0,T).
		\end{equation}
	\end{enumerate}
\end{theorem}
\begin{proof}
	The inclusion
	$\jmath_e:X_{\mathcal E}(e)\hookrightarrow \mathcal S^w(v_0)$
	is smooth by the definition of the subspace diffeology. Hence
	$\xi_X:=d(\jmath_e)_v(\xi)\in T_v\mathcal S^w(v_0).$
	By Theorem~\ref{thm:weak-distributional-linearization}, applied to
	$\xi_X$, the ambient realization $w$ satisfies
	$\operatorname{div}w=0$
	in 
	$\mathcal D'((0,T)\times\mathbb T^d),$
	and there exists
	$q\in\mathcal D'((0,T)\times\mathbb T^d)$
	such that
	$$\partial_tw
	+
	\operatorname{div}(v\otimes w+w\otimes v)
	+
	\nabla q
	=0$$
	in
	$\mathcal D'((0,T)\times\mathbb T^d;\mathbb R^d).$
	Finally, the linearized energy constraint follows from
	Proposition~\ref{prop:orthogonality}.
\end{proof}

\section{Smooth subsolutions, the lifted limit space, and the Euler locus}
\label{sec:lifted-limit-space}
This section constructs the two diffeological models used in the rest of the
paper and compares them along the locus where the flux constraint is exact.
\subsection{Smooth Euler--Reynolds subsolutions}
\label{subsec:ER-subsolutions}
We work in the variables $(v,p,R)$, where $v$ is the velocity, $p$ the
pressure, and $R$ the symmetric defect tensor, and recall
the notations
$$
\Sym(d):=\{A\in\mathbb R^{d\times d}\mid A^{T}=A\},
\qquad
\mathcal S^d_0:=\{A\in\Sym(d)\mid \operatorname{tr}A=0\},
\qquad
A^{\circ}:=A-\tfrac1d(\operatorname{tr}A)\,\mathrm{Id}.
$$
Throughout we normalize the torus so that $|\mathbb T^d|=1$, and
$[0,T]\times\mathbb T^d$ is compact, with temporal boundary
$\{0,T\}\times\mathbb T^d$.

Define the smooth ambient space
$$
\mathcal Z_{\mathrm{ER}}
:=
C^\infty([0,T]\times\mathbb T^d;\mathbb R^d)
\times
C^\infty([0,T]\times\mathbb T^d)
\times
C^\infty([0,T]\times\mathbb T^d;\mathcal S^d_0),
$$
equipped with its product diffeology; elements are denoted
$z=(v,p,R)$. Each factor is a Fréchet, hence a convenient vector space, and so
is $\mathcal Z_{\mathrm{ER}}$. The Euler--Reynolds system reads
\begin{equation}
	\label{eq:ER-vpr}
	\partial_t v+\operatorname{div}(v\otimes v)+\nabla p
	=-\operatorname{div}R,
	\qquad
	\operatorname{div}v=0 .
\end{equation}

\begin{lemma}
	\label{lem:ER-operator-smooth}
	The map
	$$
	\mathcal F\colon
	\mathcal Z_{\mathrm{ER}}
	\longrightarrow
	C^\infty([0,T]\times\mathbb T^d;\mathbb R^d)
	\times C^\infty([0,T]\times\mathbb T^d),
	\qquad
	\mathcal F(v,p,R)
	=
	\bigl(
	\partial_t v+\operatorname{div}(v\otimes v)+\nabla p+\operatorname{div}R,
	\ \operatorname{div}v
	\bigr),
	$$
	is smooth with differential
	$$
	d\mathcal F_{(v,p,R)}(w,q,S)
	=
	\bigl(
	\partial_t w+\operatorname{div}(v\otimes w+w\otimes v)+\nabla q
	+\operatorname{div}S,
	\ \operatorname{div}w
	\bigr).
	$$
\end{lemma}

\begin{proof}
	The operators $\partial_t$, $\nabla$ and $\operatorname{div}$ are
	continuous linear between the convenient vector spaces, hence
	smooth. The map $(v,w)\mapsto\operatorname{div}(v\otimes w)$ is smooth as well. The sum of these maps is smooth with differential given by
	polarization of the quadratic part; this yields the stated formula.
\end{proof}

\begin{remark}
	\label{rem:trace-free-normalization}
	Taking the defect to be trace-free is no loss of generality. For
	$R\in C^\infty([0,T]\times\mathbb T^d;\Sym(d))$ one has
	$$
	\operatorname{div}R=\operatorname{div}R^{\circ}
	+\tfrac1d\nabla(\operatorname{tr}R),
	$$
	so \eqref{eq:ER-vpr} holds for $(v,p,R)$ if and only if it holds for
	$(v,\tilde p,R^{\circ})$, where
	$$
	\tilde p
	:=
	p+\frac1d
	\Bigl(
	\operatorname{tr}R-\int_{\mathbb T^d}\operatorname{tr}R(t,\cdot)\,dx
	\Bigr).
	$$
	Subtracting the spatial mean is what preserves the normalization
	$\int_{\mathbb T^d}\tilde p(t,\cdot)\,dx=0$ of
	Definition~\ref{def:smooth-ER-subsolution}\,(3); it does not affect
	\eqref{eq:ER-vpr}, since functions of $t$ alone have vanishing spatial
	gradient.
	
	Restricting to $\mathcal S^d_0$ moreover removes a gauge redundancy in the
	pair $(p,R)$: for any $\phi\in C^\infty([0,T]\times\mathbb T^d)$ with
	$\int_{\mathbb T^d}\phi(t,\cdot)\,dx=0$, the substitution
	$(v,p,R)\mapsto(v,\,p-\phi,\,R+\phi\,\mathrm{Id})$ preserves both
	\eqref{eq:ER-vpr} and the pressure normalization, because
	$\operatorname{div}(\phi\,\mathrm{Id})=\nabla\phi$. Trace-freeness
	eliminates this one-parameter family and makes the parametrization by
	$(v,p,R)$ faithful.
\end{remark}

Fix a smooth divergence-free initial datum
$v_0\in C^\infty(\mathbb T^d;\mathbb R^d)$, $\operatorname{div}v_0=0$.

\begin{definition}
	\label{def:smooth-ER-subsolution}
	A triple $(v,p,R)\in\mathcal Z_{\mathrm{ER}}$ is a \emph{smooth
		Euler--Reynolds subsolution} with initial datum $v_0$ if
	\begin{enumerate}
		\item
		$\partial_t v+\operatorname{div}(v\otimes v)+\nabla p
		=-\operatorname{div}R$ and $\operatorname{div}v=0$ hold pointwise on
		$[0,T]\times\mathbb T^d$;
		\item
		$v(0,\cdot)=v_0$;
		\item
		$\displaystyle\int_{\mathbb T^d}p(t,x)\,dx=0$ for all $t\in[0,T]$.
	\end{enumerate}
	We denote the corresponding diffeological subspace by
	$\mathcal X_{\mathrm{ER}}^\infty(v_0)\subseteq\mathcal Z_{\mathrm{ER}}$,
	equipped with the subspace diffeology, and write
	$\iota_{\mathrm{ER}}\colon
	\mathcal X_{\mathrm{ER}}^\infty(v_0)\hookrightarrow\mathcal Z_{\mathrm{ER}}$
	for the inclusion.
\end{definition}
 
\begin{example}
	\label{ex:ER-stratum-nonempty}
	The stratum $\mathcal X_{\mathrm{ER}}^\infty(v_0)$ is never empty: the
	frozen triple
	$$
	v\equiv v_0,
	\qquad
	R:=-(v_0\otimes v_0)^{\circ},
	\qquad
	p:=-\frac1d
	\Bigl(|v_0|^{2}-\int_{\mathbb T^d}|v_0|^{2}\,dx\Bigr)
	$$
	satisfies Definition~\ref{def:smooth-ER-subsolution}. Indeed
	$\partial_t v=0$ and
	$$-\operatorname{div}R=\operatorname{div}(v_0\otimes v_0)
	-\tfrac1d\nabla|v_0|^{2}
	=\operatorname{div}(v_0\otimes v_0)+\nabla p.$$
	Thus the defect can absorb the entire nonlinearity, and no compatibility
	condition on $v_0$ beyond $\operatorname{div}v_0=0$ is needed.
\end{example}

Since $\mathcal Z_{\mathrm{ER}}$ is a convenient vector space with its
canonical diffeology, Theorem~\ref{thm:convenient-tangent}  
provides for every $z\in\mathcal Z_{\mathrm{ER}}$ a linear isomorphism
$
\Phi_z^{\mathcal Z_{\mathrm{ER}}}\colon
\mathcal Z_{\mathrm{ER}}\;\xrightarrow{\ \sim\ }\;T_z\mathcal Z_{\mathrm{ER}} .
$
For $z\in\mathcal X_{\mathrm{ER}}^\infty(v_0)$ and
$\xi\in T_z\mathcal X_{\mathrm{ER}}^\infty(v_0)$ we call
$$
\mathrm{amb}_z(\xi)
:=
\bigl(\Phi_z^{\mathcal Z_{\mathrm{ER}}}\bigr)^{-1}
\bigl(d(\iota_{\mathrm{ER}})_z(\xi)\bigr)
\in\mathcal Z_{\mathrm{ER}}
$$
the ambient realization of $\xi$; the map $\mathrm{amb}_z$ is linear as
a composition of linear maps. Finally we set
\begin{equation}
	\label{eq:lin-ER-vpr}
	\mathcal L_z
	:=
	\left\{
	(w,q,S)\in\mathcal Z_{\mathrm{ER}}
	\;\middle|\;
	\begin{aligned}
		&\partial_t w+\operatorname{div}(v\otimes w+w\otimes v)+\nabla q
		=-\operatorname{div}S,\\
		&\operatorname{div}w=0,\qquad w(0,\cdot)=0,\qquad
		\int_{\mathbb T^d}q(t,\cdot)\,dx=0 \ \ \forall t\in[0,T]
	\end{aligned}
	\right\},
\end{equation}
the linear space of solutions of the Euler--Reynolds system linearized at
$z=(v,p,R)$; by Lemma~\ref{lem:ER-operator-smooth} the first two conditions
say precisely that $(w,q,S)\in\ker d\mathcal F_z$.

\begin{proposition}
	\label{prop:linearized-ER-vpr}
	Let $z=(v,p,R)\in\mathcal X_{\mathrm{ER}}^\infty(v_0)$. Then
	$$
	\mathrm{amb}_z\bigl(T_z\mathcal X_{\mathrm{ER}}^\infty(v_0)\bigr)
	=
	\mathcal L_z .
	$$
\end{proposition}

\begin{proof}
	By Remark~\ref{rem:curve-representation-finite-sums}  
	every internal tangent vector at $z$ is a finite linear combination of
	tangent vectors represented by smooth curves through $z$. Since
	$\mathrm{amb}_z$ is linear and $\mathcal L_z$ is a linear subspace, it
	suffices to treat $\xi=\frac{d}{ds}\big|_{0}\gamma$ for a smooth curve
	$\gamma\colon(-\varepsilon,\varepsilon)\to
	\mathcal X_{\mathrm{ER}}^\infty(v_0)$,
	$\gamma(s)=(v_s,p_s,R_s)$, $\gamma(0)=z$. Naturality of the differential
	gives
	$$
	\Phi_z^{\mathcal Z_{\mathrm{ER}}}\bigl(\mathrm{amb}_z(\xi)\bigr)
	=
	\frac{d}{ds}\Big|_{0}(\iota_{\mathrm{ER}}\circ\gamma),
	$$
	hence 
	$$
	\mathrm{amb}_z(\xi)
	=
	\frac{d}{ds}\Big|_{s=0}(v_s,p_s,R_s)=:(w,q,S)
	$$
	in $\mathcal Z_{\mathrm{ER}}$, the last identification being the canonical
	one for a convenient vector space. Each $\gamma(s)$ satisfies
	$\mathcal F(\gamma(s))=(0,0)$, so differentiating at $s=0$ and using the
	chain rule together with Lemma~\ref{lem:ER-operator-smooth} yields
	$(w,q,S)\in\ker d\mathcal F_z$, i.e.\ the first two conditions in
	\eqref{eq:lin-ER-vpr}. Differentiating $v_s(0,\cdot)=v_0$ gives
	$w(0,\cdot)=0$, and differentiating $\int_{\mathbb T^d}p_s(t,x)\,dx=0$
	gives the pressure normalization.
	
    Conversely, let $(w,q,S)\in\mathcal L_z$ and set, for $s\in\mathbb R$,
	$$
	v_s:=v+sw,
	\qquad
	R_s:=R+sS-s^{2}(w\otimes w)^{\circ},
	\qquad
	p_s:=p+sq-\frac{s^{2}}{d}
	\Bigl(|w|^{2}-\int_{\mathbb T^d}|w|^{2}\,dx\Bigr).
	$$
	All three depend polynomially on $s$, so
	$s\mapsto(v_s,p_s,R_s)$ is a smooth curve in $\mathcal Z_{\mathrm{ER}}$,
	and $R_s$ takes values in $\mathcal S^d_0$ because $(\,\cdot\,)^{\circ}$
	does. Expanding the nonlinearity and cancelling the terms of order $s^{0}$
	and $s^{1}$ by means of \eqref{eq:ER-vpr} and \eqref{eq:lin-ER-vpr},
	$$
	\partial_t v_s+\operatorname{div}(v_s\otimes v_s)+\nabla(p+sq)
	=
	-\operatorname{div}R-s\operatorname{div}S
	+s^{2}\operatorname{div}(w\otimes w).
	$$
	Since
	$\operatorname{div}(w\otimes w)
	=\operatorname{div}(w\otimes w)^{\circ}+\tfrac1d\nabla|w|^{2}$,
	the right-hand side equals
	$-\operatorname{div}R_s+\tfrac{s^{2}}{d}\nabla|w|^{2}$, and absorbing the
	gradient term into the pressure gives
	$$
	\partial_t v_s+\operatorname{div}(v_s\otimes v_s)+\nabla p_s
	=-\operatorname{div}R_s .
	$$
	Moreover $\operatorname{div}v_s=\operatorname{div}v+s\operatorname{div}w=0$,
	and $v_s(0,\cdot)=v_0$ because $w(0,\cdot)=0$. Finally, using
	$\int_{\mathbb T^d}q(t,\cdot)\,dx=0$ and $|\mathbb T^d|=1$,
	$$
	\int_{\mathbb T^d}p_s(t,x)\,dx
	=
	-\frac{s^{2}}{d}
	\Bigl(
	\int_{\mathbb T^d}|w|^{2}\,dx-\int_{\mathbb T^d}|w|^{2}\,dx
	\Bigr)
	=0 .
	$$
	Hence $(v_s,p_s,R_s)\in\mathcal X_{\mathrm{ER}}^\infty(v_0)$ for every $s$,
	and $\xi:=\frac{d}{ds}\big|_{0}(v_s,p_s,R_s)
	\in T_z\mathcal X_{\mathrm{ER}}^\infty(v_0)$ satisfies
	$\mathrm{amb}_z(\xi)=(w,q,S)$.
\end{proof}
 
	The surjectivity in Proposition~\ref{prop:linearized-ER-vpr} is exactly the
	room provided by the defect: the second-order corrector
	$-s^{2}(w\otimes w)^{\circ}$ is absorbed into $R_s$, so no constraint beyond
	\eqref{eq:lin-ER-vpr} survives, and the deformation is realized by an
	honest polynomial curve rather than an asymptotic construction. Since no
	sign or cone condition is imposed on $R$, the stratum
	$\mathcal X_{\mathrm{ER}}^\infty(v_0)$ has no boundary strata.
	
	We stress what the statement does not assert. It identifies the
	image $\mathrm{amb}_z\bigl(T_z\mathcal X_{\mathrm{ER}}^\infty(v_0)\bigr)$
	with $\mathcal L_z$; injectivity of $d(\iota_{\mathrm{ER}})_z$, and hence an
	isomorphism $T_z\mathcal X_{\mathrm{ER}}^\infty(v_0)\cong\mathcal L_z$, is
	not claimed. Internal tangent spaces of diffeological subspaces may
	contain vectors that the inclusion does not detect, and all statements
	below are therefore formulated in terms of ambient realizations.

\begin{corollary} 
	\label{cor:smooth-stress-gauge}
	Let $z=(v,p,R)\in\mathcal X_{\mathrm{ER}}^\infty(v_0)$. The following two
	subspaces of $\mathcal L_z$ are independent of $z$.
	\begin{enumerate}
		\item\label{it:sg-frozen-pressure}
		\emph{(Velocity and pressure frozen.)}
		$$
		\mathcal L_z\cap
		\bigl(\{0\}\times\{0\}\times
		C^\infty([0,T]\times\mathbb T^d;\mathcal S^d_0)\bigr)
		=
		\bigl\{(0,0,S)\;:\;\operatorname{div}S=0\bigr\}.
		$$
		This space is infinite dimensional: it contains every $x$-independent
		field $S(t,x)=a(t)$ with $a\in C^\infty([0,T];\mathcal S^d_0)$, and for
		$d\ge3$ also the oscillatory modes $S(t,x)=a(t)A\cos(k\cdot x)$ with
		$k\in\mathbb Z^d\setminus\{0\}$ and $A\in\mathcal S^d_0$ satisfying
		$Ak=0$, the space of such $A$ having dimension $\tfrac{d(d-1)}2-1$. For
		$d=2$ this space is trivial, so every divergence-free field in
		$C^\infty([0,T]\times\mathbb T^2;\mathcal S^2_0)$ is $x$-independent.
		
		\item\label{it:sg-free-pressure}
		\emph{(Velocity frozen, pressure free.)}
		$$
		\mathcal L_z\cap
		\bigl(\{0\}\times C^\infty([0,T]\times\mathbb T^d)\times
		C^\infty([0,T]\times\mathbb T^d;\mathcal S^d_0)\bigr)
		=
		\left\{
		(0,q,S)
		\;\middle|\;
		\begin{aligned}
			&\operatorname{div}S+\nabla q=0,\\
			&\textstyle\int_{\mathbb T^d}q(t,\cdot)\,dx=0
			\ \ \forall t\in[0,T]
		\end{aligned}
		\right\}.
		$$
		The projection $(0,q,S)\mapsto S$ is a linear isomorphism of this space
		onto
		$$
		\mathcal G
		:=
		\bigl\{
		S\in C^\infty([0,T]\times\mathbb T^d;\mathcal S^d_0)
		\;:\;
		\partial_i(\operatorname{div}S)_j=\partial_j(\operatorname{div}S)_i
		\ \text{ for all }i,j
		\bigr\},
		$$
		with inverse
		$S\mapsto\bigl(0,-\Delta^{-1}\operatorname{div}\operatorname{div}S,S\bigr)$,
		where $\Delta^{-1}$ denotes the inverse spatial Laplacian on functions of
		vanishing spatial mean. The space in (\ref{it:sg-frozen-pressure}) is
		exactly the subspace of (\ref{it:sg-free-pressure}) cut out by $q=0$.
	\end{enumerate}
\end{corollary}

\begin{proof}
	Substituting $w=0$ into \eqref{eq:lin-ER-vpr} annihilates the only
	$z$-dependent term, $\operatorname{div}(v\otimes w+w\otimes v)$, and makes
	the constraints $\operatorname{div}w=0$ and $w(0,\cdot)=0$ vacuous; this
	proves both descriptions and their $z$-independence, part
	(\ref{it:sg-frozen-pressure}) being the case $q=0$.
	
	For the dimension count in \ref{it:sg-frozen-pressure}, an $A\in\Sym(d)$
	with $Ak=0$ restricts to a symmetric endomorphism of $k^{\perp}$, and
	$\operatorname{tr}A=0$ is one further linear condition on it, whence
	$\dim=\dim\Sym(d-1)-1=\tfrac{d(d-1)}2-1$. Such an $A$ gives
	$\operatorname{div}\bigl(A\cos(k\cdot x)\bigr)=-(Ak)\sin(k\cdot x)=0$.
	Conversely, if $S$ is divergence-free on $\mathbb T^2$ with values in
	$\mathcal S^2_0$, then each Fourier mode satisfies $\hat S(t,k)k=0$; this
	constraint is real-linear in $\hat S$, so writing $\hat S=A+iB$ with
	$A,B\in\mathcal S^2_0$ gives $Ak=Bk=0$, and the computation above forces
	$A=B=0$ for $k\neq0$.
	
	For (\ref{it:sg-free-pressure}), if $(0,q,S)$ lies in the displayed space then
	$\operatorname{div}S=-\nabla q$ is a spatial gradient, hence curl-free, so
	$S\in\mathcal G$. Conversely let $S\in\mathcal G$. Every component of
	$\operatorname{div}S$ is a spatial divergence, so
	$\widehat{\operatorname{div}S}(t,0)=0$, while curl-freeness says that
	$\widehat{\operatorname{div}S}(t,k)=\mu(t,k)\,k$ for $k\neq0$ and some
	scalars $\mu(t,k)$. Since $\operatorname{div}\operatorname{div}S$ has
	vanishing spatial mean, $q:=-\Delta^{-1}\operatorname{div}\operatorname{div}S$
	is well defined, smooth, and of vanishing spatial mean, with
	$\hat q(t,k)=i\mu(t,k)$ for $k\neq0$; hence
	$\widehat{\nabla q}(t,k)=ik\,\hat q(t,k)=-\mu(t,k)k
	=-\widehat{\operatorname{div}S}(t,k)$, i.e.\ $\operatorname{div}S+\nabla q=0$.
	Two admissible pressures differ by a function of $t$ alone with vanishing
	spatial mean, hence coincide, so the projection is bijective with the stated
	inverse. Finally, $q=0$ forces $\operatorname{div}S=0$, and conversely
	$\operatorname{div}S=0$ gives $\nabla q=0$, whence $q=0$ by the
	normalization.
\end{proof}

\subsubsection{The zero-stress Euler locus}
\label{subsubsec:zero-stress-Euler-locus}

The Euler equations appear as the zero-stress locus of the Euler--Reynolds
stratum. Define
$$
\mathcal X_{\mathrm{Euler}}^\infty(v_0)
:=
\{(v,p,R)\in\mathcal X_{\mathrm{ER}}^\infty(v_0)\mid R=0\},
$$
so that $(v,p,0)\in\mathcal X_{\mathrm{Euler}}^\infty(v_0)$ if and only if
$$
\partial_t v+\operatorname{div}(v\otimes v)+\nabla p=0,
\qquad
\operatorname{div}v=0,
\qquad
v(0,\cdot)=v_0,
\qquad
\int_{\mathbb T^d}p(t,x)\,dx=0 ,
$$
i.e.\ $\mathcal X_{\mathrm{Euler}}^\infty(v_0)$ is the normalized smooth
incompressible Euler solution space with initial datum $v_0$. Let
$\iota_{\mathrm{Euler}}\colon
\mathcal X_{\mathrm{Euler}}^\infty(v_0)
\hookrightarrow\mathcal X_{\mathrm{ER}}^\infty(v_0)$ denote the inclusion.

\begin{proposition}
	\label{pro:ER-zero-stress-tangency}
	Let $z=(v,p,0)\in\mathcal X_{\mathrm{Euler}}^\infty(v_0)$ and
	$\xi\in T_z\mathcal X_{\mathrm{Euler}}^\infty(v_0)$, and write
	$(w,q,S):=\mathrm{amb}_z\bigl(d(\iota_{\mathrm{Euler}})_z(\xi)\bigr)$.
	Then $S=0$ and
	$$
	\partial_t w+\operatorname{div}(v\otimes w+w\otimes v)+\nabla q=0,
	\qquad
	\operatorname{div}w=0,
	\qquad
	w(0,\cdot)=0,
	\qquad
	\int_{\mathbb T^d}q(t,x)\,dx=0 .
	$$
\end{proposition}

\begin{proof}
	Every plot of $\mathcal X_{\mathrm{Euler}}^\infty(v_0)$ has vanishing third
	component, so $S=0$ by Remark~\ref{rem:curve-representation-finite-sums}
	and linearity of $\mathrm{amb}_z$. The remaining assertions follow by
	substituting $S=0$ into \eqref{eq:lin-ER-vpr} and applying
	Proposition~\ref{prop:linearized-ER-vpr}.
\end{proof}
 
\begin{remark} 
	\label{rem:euler-locus-rigidity}
	Proposition~\ref{pro:ER-zero-stress-tangency} is a consistency check on the
	formalism rather than a source of nontrivial directions. Two smooth
	solutions of the incompressible Euler equations on $[0,T]\times\mathbb T^d$
	with the same initial datum coincide, and the normalization
	$\int_{\mathbb T^d}p(t,\cdot)\,dx=0$ then determines the pressure. Hence
	$\mathcal X_{\mathrm{Euler}}^\infty(v_0)$ is either empty or a single
	point, and consequently
	$T_z\mathcal X_{\mathrm{Euler}}^\infty(v_0)=\{0\}$: the conclusion of
	Proposition~\ref{pro:ER-zero-stress-tangency} is realized only by
	$(w,q,S)=(0,0,0)$.
	
	At the smooth level, with the initial datum frozen, all flexibility of
	$\mathcal X_{\mathrm{ER}}^\infty(v_0)$ is therefore carried by the defect,
	as quantified by Corollary~\ref{cor:smooth-stress-gauge}. Obtaining an
	Euler locus with nontrivial internal tangent directions requires either
	releasing $v_0$ or weakening the regularity so that uniqueness fails. The
	second route is the one taken in Subsection~\ref{subsec:lifted-limit-space-definition},
	where the flux is carried as an independent variable and the Euler locus is
	cut out inside a lifted limit space.
\end{remark}

\subsection{The lifted limit space}
\label{subsec:lifted-limit-space-definition}
We reformulate smooth strict subsolutions in a weakly closed lifted
framework. The key idea is to regard the velocity, the quadratic flux, and the
Reynolds stress as independent variables before passing to weak limits. The
exact quadratic flux constraint is then recovered as a nonlinear locus in the
lifted space.

Let
$\mathcal V_w=L^\infty(0,T;L^2(\mathbb T^d;\mathbb R^d)),$
equipped with its Banach norm topology.
Let
$\mathcal Q
=
\mathcal D'\bigl((0,T)\times\mathbb T^d;\mathbb R^{d\times d}\bigr),$
$\mathcal R
=
\mathcal D'\bigl((0,T)\times\mathbb T^d;\mathcal S^d_0\bigr),$
equipped with their usual strong distribution topologies. Each of these
is a convenient vector space, and we equip it with its canonical diffeology.
The weak-* and weak distribution topologies will be used only to define the
closure below; they are not used to define the ambient diffeology. This
distinction is essential because the quadratic map
$v\mapsto v\otimes v$ is continuous for the norm topology on $\mathcal V_w$
but not for its weak-* topology.
Define the lifted ambient space by the convenient vector space
$$\mathbf E
=
\mathcal V_w\times\mathcal Q\times\mathcal R,$$
endowed with the product diffeology. A point of $\mathbf E$ will be written as
$z=(v,Q,R),$
where $v$ is the velocity, $Q$ is the lifted quadratic flux, and $R$ is the
trace-free Reynolds stress.

For $A\in\mathcal D'\bigl((0,T)\times\mathbb T^d;\mathbb R^{m}\bigr)$ and
$\phi\in C_c^\infty\bigl((0,T)\times\mathbb T^d;\mathbb R^{m}\bigr)$ we write
$\langle A,\phi\rangle$ for the duality pairing over
$(0,T)\times\mathbb T^d$, so that
$$
\langle A,\phi\rangle
=
\int_0^T\!\!\int_{\mathbb T^d} A\cdot\phi\,dx\,dt
$$
whenever $A$ is locally integrable. For matrix fields we use
$\langle A,\nabla\phi\rangle:=\langle A_{ij},\partial_i\phi_j\rangle$, and
$\operatorname{div}$ acts as $(\operatorname{div}A)_j:=\partial_iA_{ij}$, so
that
$\langle\operatorname{div}A,\phi\rangle=-\langle A,\nabla\phi\rangle$.

Throughout, $\mathcal A\colon (0,T)\times\mathbb T^d\times\mathbb R^d\rightrightarrows\mathcal S^d_0$
is a set-valued map with, for every $(t,x,u)$,
$$
\mathcal A(t,x,u)\subseteq\mathcal S^d_0
\quad\text{nonempty, closed, convex, and}\quad
\operatorname{int}\mathcal A(t,x,u)\neq\varnothing,
$$
where the interior is taken relative to $\mathcal S^d_0$.

\begin{definition}
	\label{def:lift-smooth-subsolution}
	A triple $(v,p,R)$ is a \emph{lift-smooth strict subsolution} of the
	Euler--Reynolds system on $(0,T)\times\mathbb{T}^d$ with initial datum
	$v_0\in C^\infty(\mathbb{T}^d;\mathbb{R}^d)$ if
	\begin{enumerate}
		\item\label{it:lss-regularity}
		$v\in C^\infty([0,T]\times\mathbb{T}^d;\mathbb{R}^d)$,\quad
		$p\in C^\infty([0,T]\times\mathbb{T}^d;\mathbb{R})$,\quad
		$R\in C^\infty\!\bigl([0,T]\times\mathbb{T}^d;\,\mathcal{S}_0^d\bigr)$;
		\item\label{it:lss-equation}
		$\partial_t v + \operatorname{div}(v\otimes v) + \nabla p= -\operatorname{div}R$
		\quad and \quad
		$\operatorname{div}v = 0$
		\quad on $[0,T]\times\mathbb{T}^d$;
		\item\label{it:lss-initial}
		$v(0,\cdot) = v_0$;
		\item\label{it:lss-admissibility}
		$R(t,x)\in\operatorname{int}\mathcal{A}(t,x,v(t,x))$
		for all $(t,x)\in(0,T)\times\mathbb{T}^d$.
	\end{enumerate}
	Here the Euler--Reynolds equation and the divergence constraint are
	understood on $[0,T]\times\mathbb T^d$, the initial condition at $t=0$, and
	the admissibility condition on $(0,T)\times\mathbb T^d$, where $\mathcal A$
	is defined.
	The set of all such triples is denoted $\mathfrak S^\infty_{\mathrm{str}}(v_0,\mathcal A)$.
\end{definition}
After subtracting the spatial mean of $p$, every lift-smooth strict subsolution
admits a unique pressure normalization satisfying
$\int_{\mathbb T^d}p(t,x)\,dx=0$. This identifies the resulting normalized
triples with the subset of $\mathcal X^\infty_{\mathrm{ER}}(v_0)$ cut out by the
strict admissibility condition~\ref{it:lss-admissibility}; the normalization is
harmless because $p$ is not part of the lifted state.

The pressure is not included in the lifted state. It is recovered only up to the
usual distributional ambiguity through the weak momentum equation.

\begin{definition} 
	\label{def:lifted-limit-space}
	The lifting map is
	$$
	\Lambda:
	\mathfrak S^\infty_{\mathrm{str}}(v_0,\mathcal A)
	\longrightarrow
	\mathbf E,
	\qquad
	\Lambda(v,p,R):=(v,v\otimes v,R).
	$$
	The lifted limit space associated with $v_0$ and $\mathcal A$ is
	$$
	\mathbf X(v_0,\mathcal A)
	:=
	\overline{
		\Lambda\bigl(\mathfrak S^\infty_{\mathrm{str}}(v_0,\mathcal A)\bigr)
	}^{\,\tau_{\mathrm{cl}}},
	$$
	where $\tau_{\mathrm{cl}}$ is the product of the weak-* topology on
	$\mathcal V_w$ and the weak distribution topologies on $\mathcal Q$ and
	$\mathcal R$.
	The space $\mathbf X(v_0,\mathcal A)$ is equipped with the induced
	subspace diffeology inherited from the canonical
	diffeology of $\mathbf E$.
\end{definition}

Thus $z=(v,Q,R)\in\mathbf X(v_0,\mathcal A)$ precisely when it is a weak
lifted limit of smooth strict subsolutions:
$$
(v_\alpha,v_\alpha\otimes v_\alpha,R_\alpha)
\longrightarrow
(v,Q,R)
\qquad
\text{for the topology }\tau_{\mathrm{cl}}.
$$
The closure is taken after lifting. Consequently, the variables $Q$ and $R$
remain independent in the limit. In general one should not expect
$Q=v\otimes v$
on all of $\mathbf X(v_0,\mathcal A)$.

	The weak topology defining $\mathbf X(v_0,\mathcal A)$ does not, by
	itself, encode a strong trace at $t=0$. Thus $v_0$ is understood as part
	of the construction of the approximating smooth strict subsolutions. If one
	wants the initial datum to be a closed ambient constraint, the ambient space
	must be enriched with the appropriate trace component.

    Since $\mathbf E$ is a convenient vector space, Theorem~\ref{thm:convenient-tangent}
    provides a canonical vector-space isomorphism
    $\Phi_z^{\mathbf E}:\mathbf E\xrightarrow{\;\cong\;}T_z\mathbf E.$
    Whenever
    $\jmath:\mathbf X(v_0,\mathcal A)\hookrightarrow \mathbf E$
    denotes the inclusion of a diffeological subspace into $\mathbf E$, and
    $\xi\in T_z\mathbf X(v_0,\mathcal A),$
    we write
    $$
    \operatorname{Amb}_z(\xi)
    :=
    (\Phi_z^{\mathbf E})^{-1}\bigl((d\jmath)_z(\xi)\bigr)
    \in\mathbf E
    $$
    for the ambient realization of $\xi$. Thus, if
    $\operatorname{Amb}_z(\xi)=(w,H,S),$
    then $(w,H,S)$ is the actual lifted distributional tangent triple associated
    with $\xi$.

\begin{proposition} 
	\label{prop:weak-lifted-er}
	Let
	$z=(v,Q,R)\in\mathbf X(v_0,\mathcal A).$
	Then
	$\operatorname{div}v=0$
	in
	$\mathcal D'\bigl((0,T)\times\mathbb T^d\bigr),$
	and there exists
	$p\in\mathcal D'\bigl((0,T)\times\mathbb T^d\bigr)$
	such that
	\begin{equation}
		\label{eq:weak-lifted-er}
		\partial_t v+\operatorname{div}Q+\nabla p
		=
		-\operatorname{div}R
		\qquad
		\text{in }
		\mathcal D'\bigl((0,T)\times\mathbb T^d;\mathbb R^d\bigr).
	\end{equation}
\end{proposition}

\begin{proof}
	Let $(v_\alpha,v_\alpha\otimes v_\alpha,R_\alpha)\to(v,Q,R)$
	for $\tau_{\mathrm{cl}}$. For any test field
	$\psi\in C_c^\infty\bigl((0,T)\times\mathbb T^d;\mathbb R^d\bigr)$
	with $\operatorname{div}\psi=0$, the Euler--Reynolds equation for the
	smooth subsolutions yields
	$$
	\langle v_\alpha,\partial_t\psi\rangle
	+\langle v_\alpha\otimes v_\alpha,\nabla\psi\rangle
	+\langle R_\alpha,\nabla\psi\rangle
	=0 ,
	$$
	the pressure term dropping out because $\operatorname{div}\psi=0$.
	Passing to the limit in each factor gives
	$$
	\langle v,\partial_t\psi\rangle
	+\langle Q,\nabla\psi\rangle
	+\langle R,\nabla\psi\rangle
	=0 .
	$$
	Define $F\in\mathcal D'\bigl((0,T)\times\mathbb T^d;\mathbb R^d\bigr)$ by
	$$
	\langle F,\phi\rangle
	:=
	-\bigl(
	\langle v,\partial_t\phi\rangle
	+\langle Q,\nabla\phi\rangle
	+\langle R,\nabla\phi\rangle
	\bigr),
	\qquad
	\phi\in C_c^\infty\bigl((0,T)\times\mathbb T^d;\mathbb R^d\bigr),
	$$
	so that $F=\partial_t v+\operatorname{div}(Q+R)$ in $\mathcal D'$. The
	passage to the limit shows $\langle F,\phi\rangle=0$ for every such $\phi$
	with $\operatorname{div}\phi=0$. 
	By Lemma~\ref{lem:distributional-hodge}, $F=\nabla\pi$ for some
	$\pi\in\mathcal D'((0,T)\times\mathbb T^d)$.
	Setting $p:=-\pi$ gives \eqref{eq:weak-lifted-er}.
	Finally, $\operatorname{div}v_\alpha=0$ for every $\alpha$. Weak-*
	convergence in $\mathcal V_w$ implies distributional convergence, so
	passing to the limit gives $\operatorname{div}v=0$. 
\end{proof}

\begin{proposition} 
	\label{prop:lin-lifted-er}
	Let
	$z=(v,Q,R)\in\mathbf X(v_0,\mathcal A),$
	$\xi\in T_z\mathbf X(v_0,\mathcal A).$
	Write
	$$
	(w,H,S)
	:=
	\operatorname{Amb}_z(\xi)
	=
	(\Phi_z^{\mathbf E})^{-1}\bigl(d\jmath_z(\xi)\bigr).
	$$
	Then
	$\operatorname{div}w=0$
	in 
	$\mathcal D'\bigl((0,T)\times\mathbb T^d\bigr),$
	and there exists
	$q\in\mathcal D'\bigl((0,T)\times\mathbb T^d\bigr)$
	such that
	\begin{equation}
		\label{eq:lin-lifted-er}
		\partial_t w+\operatorname{div}H+\nabla q
		=
		-\operatorname{div}S
		\qquad
		\text{in }
		\mathcal D'\bigl((0,T)\times\mathbb T^d;\mathbb R^d\bigr).
	\end{equation}
\end{proposition}

\begin{proof}
	Let
	$\psi\in C_c^\infty\bigl((0,T)\times\mathbb T^d;\mathbb R^d\bigr)$
	with
	$\operatorname{div}\psi=0.$
	Define a continuous linear functional
	$$\ell_\psi\colon\mathbf E\to\mathbb R,
	\qquad
	\ell_\psi(v,Q,R)
	:=
	\langle v,\partial_t\psi\rangle
	+\langle Q,\nabla\psi\rangle
	+\langle R,\nabla\psi\rangle .$$
	By Proposition~\ref{prop:weak-lifted-er}, the restriction of $\ell_\psi$ to
	$\mathbf X(v_0,\mathcal A)$ vanishes identically:
	$\ell_\psi|_{\mathbf X(v_0,\mathcal A)}\equiv 0.$
	It follows from Proposition~\ref{prop:const-vanish} that for every
	$z\in \mathbf X(v_0,\mathcal A)$ and every
	$\xi\in T_z\mathbf X(v_0,\mathcal A)$,
	$$
	d\bigl(\ell_\psi|_{\mathbf X(v_0,\mathcal A)}\bigr)_z(\xi)=0.
	$$
	Applying the chain rule (Proposition~\ref{prop:chain-rule}) to 
	$\ell_\psi|_{\mathbf X(v_0,\mathcal A)}$, and using the linearity of $\ell_\psi$, we obtain
	$$
	0
	=
	d\bigl(\ell_\psi|_{\mathbf X(v_0,\mathcal A)}\bigr)_z(\xi)
	=
	\ell_\psi(w,H,S).
	$$
	Equivalently,
	$$
	\langle w,\partial_t\psi\rangle
	+\langle H,\nabla\psi\rangle
	+\langle S,\nabla\psi\rangle
	=0 .
	$$
	Hence
	$F:=\partial_t w+\operatorname{div}(H+S)$
	annihilates all compactly supported divergence-free test fields. By Lemma~\ref{lem:distributional-hodge}, there exists
	$\pi\in\mathcal D'\bigl((0,T)\times\mathbb T^d\bigr)$
	such that
	$F=\nabla\pi.$
	Setting $q:=-\pi$ gives \eqref{eq:lin-lifted-er}. Since $\operatorname{div}:\mathcal V_w\to
	\mathcal D'((0,T)\times\mathbb T^d)$ is continuous linear, hence smooth
	for the canonical diffeologies, differentiating the constraint
	$\operatorname{div}v=0$ gives
	$\operatorname{div}w=0.$
\end{proof}

\subsection{Tangency to the Euler locus}
\label{subsec:tangency-euler-locus}

\begin{definition} 
	\label{def:euler-locus}
	The Euler locus is
	$$
	\mathbf X_{\mathrm{Euler}}(v_0,\mathcal A)
	:=
	\left\{
	(v,Q,R)\in\mathbf X(v_0,\mathcal A)
	\;\middle|\;
	Q=v\otimes v
	\right\},
	$$
	equipped with the induced subspace diffeology.
\end{definition}
Here the term ``Euler locus'' refers to the locus where the lifted flux
variable agrees with the exact quadratic flux, $Q=v\otimes v$. The
stress variable $R$ is still allowed to be nonzero. This is more
precisely an exact-flux Euler--Reynolds locus; the locus of exact Euler
solutions is obtained by imposing in addition $R=0$.

If
$(v,Q,R)\in\mathbf X_{\mathrm{Euler}}(v_0,\mathcal A),$
then Proposition~\ref{prop:weak-lifted-er} gives
$$
\partial_t v+\operatorname{div}(v\otimes v)+\nabla p
=
-\operatorname{div}R
$$
in
$\mathcal D'\bigl((0,T)\times\mathbb T^d;\mathbb R^d\bigr)$
for some distributional pressure
$p\in\mathcal D'\bigl((0,T)\times\mathbb T^d\bigr).$

Let
$\iota_{\mathrm{Eu}}:
\mathbf X_{\mathrm{Euler}}(v_0,\mathcal A)
\hookrightarrow
\mathbf X(v_0,\mathcal A)$
denote the inclusion, and set
$\jmath_{\mathrm{Eu}}
:=
\jmath\circ\iota_{\mathrm{Eu}}
:
\mathbf X_{\mathrm{Euler}}(v_0,\mathcal A)
\hookrightarrow
\mathbf E.$
Recall that the Euler locus is characterized by the constraint
$Q=v\otimes v.$
 
\begin{proposition} 
	\label{prop:tangent-euler-locus}
	Let
	$z=(v,Q,R)\in\mathbf X_{\mathrm{Euler}}(v_0,\mathcal A),$
	$\xi\in T_z\mathbf X_{\mathrm{Euler}}(v_0,\mathcal A).$
	Write
	$$
	(w,H,S)
	:=
	(\Phi_z^{\mathbf E})^{-1}
	\bigl(d(\jmath_{\mathrm{Eu}})_z(\xi)\bigr).
	$$
	Then
	\begin{equation}
		\label{eq:tangent-euler-locus}
		H=v\otimes w+w\otimes v
		\qquad
		\text{in }
		\mathcal D'\bigl((0,T)\times\mathbb T^d;\mathbb R^{d\times d}\bigr).
	\end{equation}
\end{proposition}
\begin{proof}
	Define the ambient nonlinear flux-defect map
	\[
	F_{\mathbf E}\colon\mathbf E\longrightarrow
	\mathcal D'\bigl((0,T)\times\mathbb T^d;\mathbb R^{d\times d}\bigr),
	\qquad
	F_{\mathbf E}(v,Q,R):=Q-\widetilde{qu}(v),
	\]
	where $\widetilde{qu}:=I\circ qu$ is the distributional realization of the
	quadratic map $qu(v):=v\otimes v$, and $I:L^\infty\bigl(0,T;L^1(\mathbb T^d;\mathbb R^{d\times d})\bigr)\hookrightarrow
	\mathcal D'\bigl((0,T)\times\mathbb T^d;\mathbb R^{d\times d}\bigr)$ denotes the linear continuous
	embedding into $\mathcal D'\bigl((0,T)\times\mathbb T^d;\mathbb R^{d\times d}\bigr)$.
	For the norm topology chosen on $\mathcal V_w$, the bilinear map
	$(v,w)\mapsto v\otimes w$ is continuous from
	$\mathcal V_w\times\mathcal V_w$ to
	$L^\infty\bigl(0,T;L^1(\mathbb T^d;\mathbb R^{d\times d})\bigr)$. By Example~\ref{exa:quad},
	$qu$ is smooth. Since $I$ is continuous linear for the strong
	distribution topology, $\widetilde{qu}=I\circ qu$ is smooth.
	By the definition of the Euler locus, $F_{\mathbf E}$ vanishes identically on
	$\mathbf X_{\mathrm{Euler}}(v_0,\mathcal A)$, 
	\begin{equation}
		\label{eq:tangent-euler-locus-vanishing}
		F_{\mathbf E}\circ\jmath_{\mathrm{Eu}}=0,
	\end{equation}
	where
	$\jmath_{\mathrm{Eu}}\colon
	\mathbf X_{\mathrm{Euler}}(v_0,\mathcal A)\hookrightarrow\mathbf E$
	is the inclusion.
	Differentiating \eqref{eq:tangent-euler-locus-vanishing} at $z$ and evaluating
	at $\xi\in T_z\mathbf X_{\mathrm{Euler}}(v_0,\mathcal A)$, the chain rule yields
	\begin{equation}
		\label{eq:tangent-euler-locus-chain}
		0
		=
		d\bigl(F_{\mathbf E}\circ\jmath_{\mathrm{Eu}}\bigr)_z(\xi)
		=
		d(F_{\mathbf E})_z\bigl(d(\jmath_{\mathrm{Eu}})_z(\xi)\bigr).
	\end{equation}
	Here $d(F_{\mathbf E})_z$ is evaluated at the image of $\xi$ in the ambient
	tangent space $T_z\mathbf E$, not at $\xi$ itself. By the definition of
	$(w,H,S)$ in the statement,
	\begin{equation}
		\label{eq:tangent-euler-locus-amb}
		d(\jmath_{\mathrm{Eu}})_z(\xi)=\Phi^{\mathbf E}_z(w,H,S).
	\end{equation}
	It remains to compute $d(F_{\mathbf E})_z\circ\Phi^{\mathbf E}_z$. The map
	$F_{\mathbf E}$ is linear in $Q$, independent of $R$, and quadratic in $v$;
	hence it is smooth and
	\[
	d(F_{\mathbf E})_z\bigl(\Phi^{\mathbf E}_z(w,H,S)\bigr)
	=
	H-d(\widetilde{qu})_v(w).
	\]
	Since $I$ is linear and continuous, it commutes with differentiation, and
	$d(qu)_v(w)=v\otimes w+w\otimes v$, so that
	\[
	d(\widetilde{qu})_v(w)
	=
	I\bigl(d(qu)_v(w)\bigr)
	=
	v\otimes w+w\otimes v
	\qquad
	\text{in }\mathcal D'\bigl((0,T)\times\mathbb T^d;\mathbb R^{d\times d}\bigr).
	\]
	Combining this with \eqref{eq:tangent-euler-locus-chain}
	and \eqref{eq:tangent-euler-locus-amb} gives
	$0=H-(v\otimes w+w\otimes v)$, which is exactly
	\eqref{eq:tangent-euler-locus}.
\end{proof}

\begin{corollary} 
	\label{cor:lin-er-along-euler-locus}
	Let
	$z=(v,Q,R)\in\mathbf X_{\mathrm{Euler}}(v_0,\mathcal A)$,
	$\xi\in T_z\mathbf X_{\mathrm{Euler}}(v_0,\mathcal A).$
	Write
	$$
	(w,H,S)
	:=
	(\Phi_z^{\mathbf E})^{-1}
	\bigl(d(\jmath_{\mathrm{Eu}})_z(\xi)\bigr).
	$$
	Then there exists
	$q\in\mathcal D'\bigl((0,T)\times\mathbb T^d\bigr)$
	such that
	\begin{equation}
		\label{eq:lin-er-along-euler-locus}
		\partial_t w
		+
		\operatorname{div}(v\otimes w+w\otimes v)
		+
		\nabla q
		=
		-\operatorname{div}S
	\end{equation}
	in
	$\mathcal D'\bigl((0,T)\times\mathbb T^d;\mathbb R^d\bigr),$
	and
	$\operatorname{div}w=0$
	in 
	$\mathcal D'\bigl((0,T)\times\mathbb T^d\bigr).$
\end{corollary}

\begin{proof}
	Apply Proposition~\ref{prop:lin-lifted-er} to
	$d(\iota_{\mathrm{Eu}})_z(\xi)\in T_z\mathbf X(v_0,\mathcal A).$
	This yields
	$$
	\partial_t w+\operatorname{div}H+\nabla q
	=
	-\operatorname{div}S
	$$
	for some
	$q\in\mathcal D'\bigl((0,T)\times\mathbb T^d\bigr),$
	together with
	$\operatorname{div}w=0.$
	Using Proposition~\ref{prop:tangent-euler-locus}, we substitute
	$H=v\otimes w+w\otimes v$
	and obtain \eqref{eq:lin-er-along-euler-locus}.
\end{proof}
 
\subsection{Observable kernels}
\label{subsec:observable-kernels}

We now study the infinitesimal separation properties of natural linear probe
observables on the lifted limit space
$\mathbf X(v_0,\mathcal A)\subseteq \mathbf E$.
Let
$\jmath:\mathbf X(v_0,\mathcal A)\hookrightarrow \mathbf E$
denote the canonical inclusion. Since $\mathbf X(v_0,\mathcal A)$ carries the
subspace diffeology, the map $\jmath$ is smooth. Moreover, every observable
considered below is the restriction to $\mathbf X(v_0,\mathcal A)$ of a
continuous linear functional on the ambient weak space $\mathbf E$; hence these
observables are smooth.

For
$z\in \mathbf X(v_0,\mathcal A)$
and
$\xi\in T_z\mathbf X(v_0,\mathcal A)$,
we use the ambient realization convention
$$
\operatorname{Amb}_z(\xi)
:=
(\Phi_z^{\mathbf E})^{-1}\bigl(d\jmath_z(\xi)\bigr)
\in \mathbf E.
$$
Whenever
$\operatorname{Amb}_z(\xi)=(w,H,S),$
the components $w,H,S$ are understood as the velocity, flux, and stress
components of the ambient tangent realization.

\begin{definition}
	\label{def:observable-kernel}
	Let
	$z\in \mathbf X(v_0,\mathcal A)$,
	and let $\mathcal F$ be a family of real-valued observables on
	$\mathbf X(v_0,\mathcal A)$.
	The \emph{observable kernel} of $\mathcal F$ at $z$ is
	$$
	\ker_z(\mathcal F)
	:=
	\left\{
	\xi\in T_z\mathbf X(v_0,\mathcal A)
	\;\middle|\;
	dF_z(\xi)=0
	\quad\text{for every }F\in\mathcal F
	\right\}.
	$$
\end{definition}

We consider the following nested families of linear probe observables:
$$
\mathcal F_{\mathrm{vel}}
:=
\left\{
(v,Q,R)\longmapsto
\int_0^T\!\!\int_{\mathbb T^d}
v\cdot\psi\,dx\,dt
\;\middle|\;
\psi\in C_c^\infty\!\bigl((0,T)\times\mathbb{T}^d;\mathbb{R}^d\bigr)
\right\},
$$
$$
\mathcal F_{\mathrm{flux}}
:=
\mathcal F_{\mathrm{vel}}
\cup
\left\{
(v,Q,R)\longmapsto
\langle Q,\Psi\rangle
\;\middle|\;
\Psi\in
C_c^\infty\bigl((0,T)\times\mathbb T^d;\mathbb R^{d\times d}\bigr)
\right\},
$$
and
$$
\mathcal F_{\mathrm{full}}
:=
\mathcal F_{\mathrm{flux}}
\cup
\left\{
(v,Q,R)\longmapsto
\langle R,\Theta\rangle
\;\middle|\;
\Theta\in
C_c^\infty\bigl((0,T)\times\mathbb T^d;\mathcal S^d_0\bigr)
\right\}.
$$

\begin{theorem}
	\label{thm:kernel-characterization}
	Let
	$z=(v,Q,R)\in \mathbf X(v_0,\mathcal A)$,
	$\xi\in T_z\mathbf X(v_0,\mathcal A)$,
	and write
	$\operatorname{Amb}_z(\xi)=(w,H,S).$
	Then
	\begin{align*}
		\ker_z(\mathcal F_{\mathrm{vel}})
		&=
		\left\{
		\xi\in T_z\mathbf X(v_0,\mathcal A)
		\;\middle|\;
		w=0
		\right\},
		\\
		\ker_z(\mathcal F_{\mathrm{flux}})
		&=
		\left\{
		\xi\in T_z\mathbf X(v_0,\mathcal A)
		\;\middle|\;
		w=0,\ H=0
		\right\},
		\\
		\ker_z(\mathcal F_{\mathrm{full}})
		&=
		\left\{
		\xi\in T_z\mathbf X(v_0,\mathcal A)
		\;\middle|\;
		w=0,\ H=0,\ S=0
		\right\}.
	\end{align*}
	In particular,
	$$
	\ker_z(\mathcal F_{\mathrm{full}})
	=
	\ker(d\jmath_z).
	$$
	Thus the three observable families detect, respectively, the velocity
	component, the pair $(v,Q)$, and the full ambient realization, modulo the
	possible noninjectivity of the tangent map $d\jmath_z$.
\end{theorem}

\begin{proof}
	Let
	$\operatorname{Amb}_z(\xi)=(w,H,S).$
	Since all observables under consideration are restrictions of ambient linear
	functionals on $\mathbf E$, their differentials at $z$ are obtained by pairing
	$d\jmath_z(\xi)$, equivalently $\operatorname{Amb}_z(\xi)$, with the
	corresponding test fields.
	
	For
	$\psi\in C_c^\infty\!\bigl((0,T)\times\mathbb{T}^d;\mathbb{R}^d\bigr)$,
	the corresponding velocity observable satisfies
	$$
	d\!\left(
	(v,Q,R)\mapsto
	\int_0^T\!\!\int_{\mathbb T^d} v\cdot\psi\,dx\,dt
	\right)_z(\xi)
	=
	\int_0^T\!\!\int_{\mathbb T^d} w\cdot\psi\,dx\,dt.
	$$
	Hence all velocity probes vanish on $\xi$ if and only if $w=0$.

	Similarly, for
	$\Psi\in C_c^\infty\bigl((0,T)\times\mathbb T^d;\mathbb R^{d\times d}\bigr)$
	and
	$\Theta\in C_c^\infty\bigl((0,T)\times\mathbb T^d;\mathcal S^d_0\bigr)$,
	we have
	$$
	d\langle Q,\Psi\rangle_z(\xi)=\langle H,\Psi\rangle,
	\qquad
	d\langle R,\Theta\rangle_z(\xi)=\langle S,\Theta\rangle.
	$$
	Since smooth compactly supported test tensors separate distributions, the
	vanishing of all flux probes is equivalent to $H=0$, and the vanishing of all
	stress probes is equivalent to $S=0$. This proves the three kernel
	characterizations.
	Finally,
	$$
	\ker_z(\mathcal F_{\mathrm{full}})
	=
	\left\{
	\xi\in T_z\mathbf X(v_0,\mathcal A)
	\;\middle|\;
	\operatorname{Amb}_z(\xi)=(0,0,0)
	\right\}.
	$$
	By definition,
	$\operatorname{Amb}_z(\xi)
	=
	(\Phi_z^{\mathbf E})^{-1}\bigl(d\jmath_z(\xi)\bigr),$
	so
	$$
	\operatorname{Amb}_z(\xi)=0
	\quad\Longleftrightarrow\quad
	d\jmath_z(\xi)=0.
	$$
	Therefore
	$\ker_z(\mathcal F_{\mathrm{full}})
	=
	\ker(d\jmath_z).$
\end{proof}

\begin{corollary}
	\label{cor:nested-kernels}
	For every
	$z\in \mathbf X(v_0,\mathcal A)$,
	one has
	$$
	\ker_z(\mathcal F_{\mathrm{full}})
	\subseteq
	\ker_z(\mathcal F_{\mathrm{flux}})
	\subseteq
	\ker_z(\mathcal F_{\mathrm{vel}}).
	$$
	The first inclusion is strict whenever there exists
	$\xi\in T_z\mathbf X(v_0,\mathcal A)$
	with
	$$
	\operatorname{Amb}_z(\xi)=(0,0,S),
	\qquad
	S\neq 0.
	$$
	The second inclusion is strict whenever there exists
	$\xi\in T_z\mathbf X(v_0,\mathcal A)$
	with
	$$
	\operatorname{Amb}_z(\xi)=(0,H,S),
	\qquad
	H\neq 0.
	$$
\end{corollary}

\begin{proof}
	The inclusions follow immediately from
	Theorem~\ref{thm:kernel-characterization}, and the strictness assertions follow
	from the displayed nontrivial ambient realizations.
\end{proof}

\subsubsection{Stress-gauge directions}
\label{subsubsec:stress-gauge-directions}

The residual kernel of the velocity--flux observables is encoded by the stress
component of the ambient tangent realization.

\begin{definition}
	\label{def:stress-gauge-direction}
	Let
	$z\in \mathbf X(v_0,\mathcal A)$.
	A tangent vector
	$\xi\in T_z\mathbf X(v_0,\mathcal A)$
	is called a \emph{stress-gauge direction} if
	$\operatorname{Amb}_z(\xi)=(0,0,S)$
	for some
	$S\in
	\mathcal D'\bigl((0,T)\times\mathbb T^d;\mathcal S^d_0\bigr).$
\end{definition}

\begin{theorem}
	\label{thm:stress-gauge-characterization}
	Let
	$z=(v,Q,R)\in \mathbf X(v_0,\mathcal A)$,
	$\xi\in T_z\mathbf X(v_0,\mathcal A)$,
	and write
	$\operatorname{Amb}_z(\xi)=(w,H,S).$
	Then the following are equivalent:
	\begin{enumerate}
		\item
		$\xi$ is a stress-gauge direction;
		\item
		$\xi\in \ker_z(\mathcal F_{\mathrm{flux}})$;
		\item
		$w=0$ and $H=0$.
	\end{enumerate}
	If these conditions hold, then there exists
	$\pi\in \mathcal D'\bigl((0,T)\times\mathbb T^d\bigr)$
	such that
	\begin{equation}
		\label{eq:stress-gauge-gradient-div}
		\operatorname{div}S=\nabla\pi
		\qquad
		\text{in }
		\mathcal D'\bigl((0,T)\times\mathbb T^d;\mathbb R^d\bigr).
	\end{equation}
	Moreover,
	$$
	\xi\in \ker_z(\mathcal F_{\mathrm{full}})
	\quad\Longleftrightarrow\quad
	S=0.
	$$
\end{theorem}

\begin{proof}
	The equivalence of the first three conditions follows from
	Definition~\ref{def:stress-gauge-direction} and
	Theorem~\ref{thm:kernel-characterization}.
	By Proposition~\ref{prop:lin-lifted-er}, the ambient realization satisfies
	$$
	\partial_t w+\operatorname{div}H+\nabla q
	=
	-\operatorname{div}S
	$$
	for some
	$q\in\mathcal D'\bigl((0,T)\times\mathbb T^d\bigr).$
	If $w=0$ and $H=0$, then
	$\nabla q=-\operatorname{div}S.$
	Setting
	$\pi=-q$
	gives
	$\operatorname{div}S=\nabla\pi.$
	The last assertion follows again from
	Theorem~\ref{thm:kernel-characterization}.
\end{proof}

\begin{proposition}
	\label{prop:explicit-stress-gauge-tensors}
	Let
	$\varphi\in C^\infty\bigl((0,T)\times\mathbb T^d\bigr)$,
	and define
	$S_{ij}
	:=
	\partial_i\partial_j\varphi
	-\frac1d(\Delta\varphi)\delta_{ij}.$
	Then
	$S\in
	C^\infty\bigl((0,T)\times\mathbb T^d;\mathcal S^d_0\bigr)$
	and
	$$
	\operatorname{div}S=\nabla\pi,
	\qquad
	\pi=\frac{d-1}{d}\Delta\varphi.
	$$
\end{proposition}

\begin{proof}
	Symmetry is immediate. Also,
	$$
	\operatorname{tr}S
	=
	\Delta\varphi
	-\frac1d(\Delta\varphi)\delta_{ii}
	=
	\Delta\varphi-\Delta\varphi
	=
	0,
	$$
	so $S$ takes values in $\mathcal S^d_0$. Finally,
	$$
	(\operatorname{div}S)_i
	=
	\partial_jS_{ij}
	=
	\partial_i\Delta\varphi
	-\frac1d\partial_i\Delta\varphi
	=
	\frac{d-1}{d}\partial_i\Delta\varphi.
	$$
	Hence
	$$
	\operatorname{div}S
	=
	\nabla\left(\frac{d-1}{d}\Delta\varphi\right),
	$$
	as claimed.
\end{proof}

\begin{remark}
	\label{rem:explicit-stress-gauge-vs-smooth}
	Since $\Delta\varphi$ has vanishing spatial mean, the tensor $S$ of
	Proposition~\ref{prop:explicit-stress-gauge-tensors} together with
	$q:=-\tfrac{d-1}{d}\Delta\varphi$ lies in the space of
	Corollary~\ref{cor:smooth-stress-gauge}\,(\ref{it:sg-free-pressure}), and it
	lies in (\ref{it:sg-frozen-pressure}) only when $\Delta\varphi$ is
	$x$-independent. This is the precise sense in which such tensors are smooth
	prototypes of stress-gauge directions: they solve the necessary condition
	\eqref{eq:stress-gauge-gradient-div}, but no claim is made here that they are
	realized by an actual tangent vector of $\mathbf X(v_0,\mathcal A)$.
\end{remark}

\subsubsection{Euler-locus consequences}
\label{subsubsec:euler-locus-observable-kernels}

We now restrict the preceding kernel analysis to the Euler locus
$\mathbf X_{\mathrm{Euler}}(v_0,\mathcal A)$.

\begin{corollary}
	\label{cor:vel-flux-kernels-coincide-euler}
	Let
	$z\in \mathbf X_{\mathrm{Euler}}(v_0,\mathcal A)$.
	Then
	$$
	\ker_z(\mathcal F_{\mathrm{vel}})
	\cap
	T_z\mathbf X_{\mathrm{Euler}}(v_0,\mathcal A)
	=
	\ker_z(\mathcal F_{\mathrm{flux}})
	\cap
	T_z\mathbf X_{\mathrm{Euler}}(v_0,\mathcal A).
	$$
	Equivalently, on Euler tangents the velocity observables already determine the
	flux component.
\end{corollary}
\begin{proof}
	By Theorem~\ref{thm:kernel-characterization},
	$$
	\xi\in\ker_z(\mathcal F_{\mathrm{vel}})
	\quad\Longleftrightarrow\quad
	w=0.
	$$
	For Euler tangents, Proposition~\ref{prop:tangent-euler-locus} gives
	$H=v\otimes w+w\otimes v.$
	Hence $w=0$ implies $H=0$. Using again
	Theorem~\ref{thm:kernel-characterization}, we conclude that
	$\xi\in\ker_z(\mathcal F_{\mathrm{flux}}).$
	The reverse inclusion is immediate from
	$\ker_z(\mathcal F_{\mathrm{flux}})
	\subseteq
	\ker_z(\mathcal F_{\mathrm{vel}}),$
	which holds on all of $\mathbf X(v_0,\mathcal A)$ by
	Theorem~\ref{thm:kernel-characterization}. This proves the equality.
\end{proof}
 
\begin{corollary}
	\label{cor:full-separation-euler}
	Let
	$z\in \mathbf X_{\mathrm{Euler}}(v_0,\mathcal A)$,
	and let
	$\jmath_{\mathrm{Eu}}:
	\mathbf X_{\mathrm{Euler}}(v_0,\mathcal A)
	\hookrightarrow
	\mathbf E$
	denote the canonical inclusion.
	Then
	$$
	\ker_z(\mathcal F_{\mathrm{full}})
	\cap
	T_z\mathbf X_{\mathrm{Euler}}(v_0,\mathcal A)
	=
	\ker\bigl(d(\jmath_{\mathrm{Eu}})_z\bigr).
	$$
	In particular, if
	$d(\jmath_{\mathrm{Eu}})_z$
	is injective, then
	$\ker_z(\mathcal F_{\mathrm{full}})
	\cap
	T_z\mathbf X_{\mathrm{Euler}}(v_0,\mathcal A)
	=
	\{0\}.$
\end{corollary}

\begin{proof}
	This is the restriction of
	Theorem~\ref{thm:kernel-characterization} to the Euler tangent space, with
	$\jmath_{\mathrm{Eu}}$ in place of $\jmath$. Explicitly, a tangent vector
	$\xi\in T_z\mathbf X_{\mathrm{Euler}}(v_0,\mathcal A)$
	lies in the left-hand side if and only if its ambient realization vanishes, and
	this is equivalent to
	$d(\jmath_{\mathrm{Eu}})_z(\xi)=0.$
	The final assertion is immediate.
\end{proof}

\section{A mixture viewpoint on Euler--Reynolds}
\label{sec:mixtures}

In this section we give a concrete finite-dimensional realization of the
observable kernels introduced in Section~\ref{sec:lifted-limit-space}. The
purpose is not to introduce a new completion or a new convex integration
scheme, but to exhibit explicit smooth maps in the lifted
Euler--Reynolds ambient space, compute their differentials, and relate the
resulting directions to velocity, flux, and full lifted observability.

\subsection{Finite mixtures in the lifted ambient variables}
\label{subsec:finite-mixtures}

Let
$x_i=(v_i,Q_i,R_i)\in\mathbf E,$
$1\leq i\leq N,$
be a finite family of lifted Euler--Reynolds states satisfying
$\operatorname{div}v_i=0$
and
$$
\partial_t v_i+\operatorname{div}Q_i+\nabla p_i
=
-\operatorname{div}R_i
$$
for some
$p_i\in\mathcal D'\bigl((0,T)\times\mathbb T^d\bigr).$
Two regimes must be carefully distinguished throughout this section.
In the general regime the flux $Q_i$ is an independent ambient variable,
unrelated to $v_i$; in the exact regime one has
$$
x_i=\Lambda(v_i,p_i,R_i)
=
(v_i,v_i\otimes v_i,R_i),
$$
as is the case for smooth strict subsolutions. Only in the exact regime may the
flux be interpreted as a second moment, and only then does the quadratic
covariance defect introduced in \eqref{eq:mixture-covariance-defect} below carry
its probabilistic meaning.

Let
$$
\Delta_N
:=
\left\{
\lambda=(\lambda_1,\ldots,\lambda_N)\in\mathbb R^N
\;\middle|\;
\lambda_i\geq0,\quad
\sum_{i=1}^N\lambda_i=1
\right\}
$$
be the standard simplex, and let $\Delta_N^\circ$ denote its relative
interior. For $\lambda\in\Delta_N$, define
$$
v[\lambda]:=\sum_{i=1}^N\lambda_i v_i,
\qquad
Q[\lambda]:=\sum_{i=1}^N\lambda_i Q_i,
\qquad
R[\lambda]:=\sum_{i=1}^N\lambda_i R_i.
$$

\begin{definition} 
	\label{def:lifted-mixture-map}
	The map
	$$
	P:\Delta_N\longrightarrow\mathbf E,
	\qquad
	P(\lambda):=(v[\lambda],Q[\lambda],R[\lambda]),
	$$
	is called the \emph{lifted ambient mixture map} associated with the family
	$\{x_i\}_{i=1}^N$.
\end{definition}

Set
$$
p[\lambda]:=\sum_{i=1}^N\lambda_i p_i.
$$
By linearity of the lifted Euler--Reynolds identities,
$$
\operatorname{div}v[\lambda]=0
$$
and
$$
\partial_t v[\lambda]
+\operatorname{div}Q[\lambda]
+\nabla p[\lambda]
=
-\operatorname{div}R[\lambda].
$$

\begin{remark}
	\label{rem:mixture-ambient-only}
	The mixture map $P$ is, a priori, a map into the lifted ambient space
	$\mathbf E$. Even if each
	$x_i\in\mathbf X(v_0,\mathcal A),$
	the barycentric combination $P(\lambda)$ need not belong to
	$\mathbf X(v_0,\mathcal A)$, because the lifted limit space need not be
	convex.
	If, for the relevant values of $\lambda$, one additionally knows that
	$P(\lambda)\in\mathbf X(v_0,\mathcal A),$
	then $P$ may be regarded as a map into the lifted limit space, and its
	differential gives actual tangent directions there. Without this additional
	realizability assumption, $P$ should be interpreted only as a
	finite-dimensional family in the linear lifted Euler--Reynolds ambient
	model. Sufficient conditions are given in
	Subsection~\ref{subsec:mixture-realizability}.
\end{remark}

Suppose now that each phase lies in the exact Euler locus, so that
$Q_i=v_i\otimes v_i.$
Then
$$
Q[\lambda]
=
\sum_{i=1}^N\lambda_i v_i\otimes v_i.
$$
The quadratic fluctuation of the mixture is
\begin{equation}
	\label{eq:mixture-covariance-defect}
	C[\lambda]
	:=
	Q[\lambda]-v[\lambda]\otimes v[\lambda]
	=
	\sum_{i=1}^N\lambda_i
	\bigl(v_i-v[\lambda]\bigr)\otimes\bigl(v_i-v[\lambda]\bigr).
\end{equation}
Its trace-free part is
$$
\operatorname{dev}C[\lambda]
:=
C[\lambda]
-\frac1d\bigl(\operatorname{tr}C[\lambda]\bigr)I.
$$
Thus $\operatorname{dev}C[\lambda]$ measures the deviatoric part of the
quadratic covariance defect.

\begin{lemma}
	\label{lem:mixture-covariance-sign}
	Let $Q_i=v_i\otimes v_i$ for all $i$ and let $\lambda\in\Delta_N$. Then, for
	a.e.\ $(t,x)\in(0,T)\times\mathbb T^d$,
	$$
	C[\lambda](t,x)\geq0,
	\qquad
	\operatorname{tr}C[\lambda](t,x)
	=
	\sum_{i=1}^N\lambda_i\bigl|v_i(t,x)-v[\lambda](t,x)\bigr|^2
	\geq0,
	$$
	and
	$$
	\bigl|\operatorname{dev}C[\lambda]\bigr|
	\leq
	\bigl|C[\lambda]\bigr|
	\leq
	\operatorname{tr}C[\lambda]
	\leq
	\frac12\max_{1\leq i,j\leq N}\bigl|v_i-v_j\bigr|^2,
	$$
	where $|\cdot|$ denotes the Frobenius norm. Moreover, setting
	\begin{equation}
		\label{eq:absorbed-mixture}
		\widetilde R[\lambda]
		:=
		R[\lambda]+\operatorname{dev}C[\lambda]
		\in\mathcal S^d_0,
		\qquad
		p_{\mathrm{rec}}[\lambda]
		:=
		p[\lambda]+\frac1d\operatorname{tr}C[\lambda],
	\end{equation}
	one has
	$$
	\partial_t v[\lambda]
	+\operatorname{div}\bigl(v[\lambda]\otimes v[\lambda]\bigr)
	+\nabla p_{\mathrm{rec}}[\lambda]
	=
	-\operatorname{div}\widetilde R[\lambda]
	$$
	in $\mathcal D'\bigl((0,T)\times\mathbb T^d\bigr)$, that is,
	$$
	\Lambda\bigl(v[\lambda],p_{\mathrm{rec}}[\lambda],\widetilde R[\lambda]\bigr)
	=
	\bigl(v[\lambda],\,v[\lambda]\otimes v[\lambda],\,\widetilde R[\lambda]\bigr)
	$$
	is a lifted Euler--Reynolds state.
\end{lemma}

\begin{proof}
	The identity in \eqref{eq:mixture-covariance-defect} is the standard
	variance decomposition, and it exhibits $C[\lambda]$ as a nonnegative
	combination of rank-one nonnegative matrices, whence
	$C[\lambda]\geq0$. For a nonnegative symmetric matrix the Frobenius norm is
	dominated by the trace, and $\operatorname{dev}$ is the orthogonal
	projection of $\operatorname{Sym}(d)$ onto $\mathcal S^d_0$ with respect to
	the Frobenius inner product, so it is norm-nonincreasing. The bound by the
	pairwise oscillation follows from
	$$
	\sum_{i=1}^N\lambda_i\bigl|v_i-v[\lambda]\bigr|^2
	=
	\frac12\sum_{i,j=1}^N\lambda_i\lambda_j\bigl|v_i-v_j\bigr|^2 .
	$$
	For the last assertion, subtract
	$\operatorname{div}\bigl(v[\lambda]\otimes v[\lambda]\bigr)$
	from
	$\operatorname{div}Q[\lambda]$
	and split $C[\lambda]$ into its trace-free part and its trace part, using
	$\operatorname{div}\bigl(\tfrac1d(\operatorname{tr}C[\lambda])I\bigr)
	=
	\nabla\bigl(\tfrac1d\operatorname{tr}C[\lambda]\bigr).$
	All manipulations are distributional pairings against
	$C_c^\infty\bigl((0,T)\times\mathbb T^d;\mathbb R^d\bigr)$
	and require no regularity beyond
	$C[\lambda]\in L^\infty\bigl(0,T;L^1\bigr)$,
	which holds since $v_i\in\mathcal V_w$.
\end{proof}

\begin{remark}
	\label{rem:two-mixture-points}
	Lemma~\ref{lem:mixture-covariance-sign} produces a point of $\mathbf E$
	which is not the mixture point $P(\lambda)$: the two differ in the
	flux component, since
	$Q[\lambda]=v[\lambda]\otimes v[\lambda]+C[\lambda]$
	with $C[\lambda]\neq0$ unless all phases coincide. The mixture point
	$P(\lambda)$ lies off the exact Euler locus and its membership in
	$\mathbf X(v_0,\mathcal A)$ is a genuine relaxation statement, whereas the
	absorbed point
	$\Lambda\bigl(v[\lambda],p_{\mathrm{rec}}[\lambda],\widetilde R[\lambda]\bigr)$
	lies in the image of the lifting map by construction. Note also that
	$C[\lambda]\geq0$ does not imply any sign condition on
	$\operatorname{dev}C[\lambda]$: a trace-free nonnegative matrix vanishes
	identically, so the deviatoric covariance defect is never sign-definite
	unless it is zero. Accordingly, $\operatorname{dev}C[\lambda]$ must be
	treated as an admissibility-constrained element of $\mathcal S^d_0$ in the
	sense of Section~\ref{sec:lifted-limit-space}, and not as a nonnegative
	Reynolds stress.
\end{remark}

The simplex $\Delta_N$ is equipped with the subspace diffeology inherited
from $\mathbb R^N$.

\begin{proposition} 
	\label{prop:mixture-smoothness}
	The lifted ambient mixture map
	$P:\Delta_N\longrightarrow\mathbf E$
	is smooth.  
\end{proposition}

\begin{proof}
	Each component of $P$ is a finite affine combination of fixed elements of
	the corresponding locally convex factor of $\mathbf E$. Hence $P$ is
	smooth as a map into $\mathbf E$. 
\end{proof}

For every $\lambda\in\Delta_N$, the tangent space to the affine hull of the
simplex is
$$
T_\lambda\Delta_N
=
\left\{
\delta\lambda=(\delta\lambda_1,\ldots,\delta\lambda_N)\in\mathbb R^N
\;\middle|\;
\sum_{i=1}^N\delta\lambda_i=0
\right\}.
$$
At points of $\Delta_N^\circ$, this is the ordinary tangent space of the
manifold-with-corners simplex. The differential of $P$ is
$$
dP_\lambda(\delta\lambda)
=
(\delta v,\delta Q,\delta R),
$$
where
\begin{equation}
	\label{eq:mixture-differential}
	\delta v=\sum_{i=1}^N\delta\lambda_i v_i,
	\qquad
	\delta Q=\sum_{i=1}^N\delta\lambda_i Q_i,
	\qquad
	\delta R=\sum_{i=1}^N\delta\lambda_i R_i.
\end{equation}

\begin{lemma}
	\label{lem:first-variation-dev-C}
	Let $Q_i=v_i\otimes v_i$ for all $i$, let $\lambda^0\in\Delta_N^\circ$, and
	let $\delta\lambda\in T_\lambda\Delta_N$ satisfy
	$\sum_{i=1}^N\delta\lambda_i v_i=0.$
	Set $\lambda(s):=\lambda^0+s\,\delta\lambda.$ Then
	$$
	\frac{d}{ds}\bigg|_{s=0}\operatorname{dev}C[\lambda(s)]
	=
	\operatorname{dev}\left(
	\sum_{i=1}^N\delta\lambda_i\,v_i\otimes v_i
	\right),
	$$
	and correspondingly
	$$
	\frac{d}{ds}\bigg|_{s=0}\widetilde R[\lambda(s)]
	=
	\sum_{i=1}^N\delta\lambda_i R_i
	+
	\operatorname{dev}\left(
	\sum_{i=1}^N\delta\lambda_i\,v_i\otimes v_i
	\right).
	$$
\end{lemma}

\begin{proof}
	By \eqref{eq:mixture-covariance-defect},
	$$
	C[\lambda(s)]
	=
	\sum_{i=1}^N\lambda_i(s)\,v_i\otimes v_i
	-
	v[\lambda(s)]\otimes v[\lambda(s)].
	$$
	Differentiating at $s=0$ and using
	$\tfrac{d}{ds}v[\lambda(s)]=\sum_i\delta\lambda_i v_i=0$
	kills the second term, leaving
	$\sum_i\delta\lambda_i\,v_i\otimes v_i.$
	Applying the linear map $\operatorname{dev}$ and
	\eqref{eq:absorbed-mixture} gives both identities.
\end{proof}

Formula~\eqref{eq:mixture-differential} gives a finite-dimensional analogue of
the observable kernels from Subsection~\ref{subsec:observable-kernels}. Define
$$
K_{\mathrm{vel}}(\lambda)
:=
\left\{
\delta\lambda\in T_\lambda\Delta_N
\;\middle|\;
\sum_{i=1}^N\delta\lambda_i v_i=0
\right\},
$$
$$
K_{\mathrm{flux}}(\lambda)
:=
\left\{
\delta\lambda\in K_{\mathrm{vel}}(\lambda)
\;\middle|\;
\sum_{i=1}^N\delta\lambda_i Q_i=0
\right\},
$$
and
$$
K_{\mathrm{full}}(\lambda)
:=
\left\{
\delta\lambda\in K_{\mathrm{flux}}(\lambda)
\;\middle|\;
\sum_{i=1}^N\delta\lambda_i R_i=0
\right\}.
$$

\begin{proposition}
	\label{prop:mixture-kernels}
	Let
	$\lambda\in\Delta_N^\circ$, let
	$\delta\lambda\in T_\lambda\Delta_N$, and set
	$\delta X
	:=
	dP_\lambda(\delta\lambda)
	=
	(\delta v,\delta Q,\delta R).$
	Then:
	\begin{enumerate}
		\item if
		$\delta\lambda\in K_{\mathrm{vel}}(\lambda)$, then
		$\delta X$ is annihilated by all ambient linear velocity probes;
		
		\item if
		$\delta\lambda\in K_{\mathrm{flux}}(\lambda)$, then
		$\delta X$ is annihilated by all ambient linear velocity and flux
		probes;
		
		\item if
		$\delta\lambda\in K_{\mathrm{full}}(\lambda)$, then
		$\delta X=0$
		in the lifted ambient space $\mathbf E$.
	\end{enumerate}
	Moreover, if
	$\delta\lambda\in K_{\mathrm{flux}}(\lambda)$
	and
	$\sum_{i=1}^N\delta\lambda_iR_i\neq0,$
	then
	$$
	dP_\lambda(\delta\lambda)
	=
	\left(
	0,0,\sum_{i=1}^N\delta\lambda_iR_i
	\right)
	$$
	is a nonzero ambient stress-gauge variation in the sense of
	Theorem~\ref{thm:stress-gauge-characterization}. It represents a genuine stress-gauge tangent
	direction in $\mathbf X(v_0,\mathcal A)$ only under the additional
	realizability condition that the corresponding mixture path is contained in
	$\mathbf X(v_0,\mathcal A)$.
\end{proposition}

\begin{proof}
	By the differential formula \eqref{eq:mixture-differential},
	$$
	\delta v
	=
	\sum_{i=1}^N\delta\lambda_i v_i,
	\qquad
	\delta Q
	=
	\sum_{i=1}^N\delta\lambda_i Q_i,
	\qquad
	\delta R
	=
	\sum_{i=1}^N\delta\lambda_i R_i.
	$$
	Hence
	$$
	\delta v=0
	\quad\Longleftrightarrow\quad
	\delta\lambda\in K_{\mathrm{vel}}(\lambda),
	$$
	and
	$$
	\delta v=\delta Q=0
	\quad\Longleftrightarrow\quad
	\delta\lambda\in K_{\mathrm{flux}}(\lambda).
	$$
	The first two assertions therefore follow from the fact that the velocity
	and flux probes are precisely the linear probes of the corresponding lifted
	ambient components.
	
	If
	$\delta\lambda\in K_{\mathrm{full}}(\lambda)$, then, by definition,
	$\delta v=\delta Q=\delta R=0.$
	Thus
	$dP_\lambda(\delta\lambda)=0$
	in $\mathbf E$.
	Finally, if
	$\delta\lambda\in K_{\mathrm{flux}}(\lambda)$ and
	$\sum_{i=1}^N\delta\lambda_iR_i\neq0$, then
	$$
	\delta v=0,
	\qquad
	\delta Q=0,
	\qquad
	\delta R=\sum_{i=1}^N\delta\lambda_iR_i\neq0.
	$$
	Consequently
	$dP_\lambda(\delta\lambda)=(0,0,\delta R)$
	is a nonzero variation purely in the lifted stress component. This is an
	ambient stress-gauge variation. It is a stress-gauge tangent direction in
	the lifted limit space only if the variation is realized by a plot in
	$\mathbf X(v_0,\mathcal A)$.
\end{proof}

\begin{remark}
	At the finite-mixture level,
	$K_{\mathrm{full}}(\lambda)=\ker(dP_\lambda)$ is the kernel of the explicit
	affine map $P:\Delta_N\to\mathbf E$. This must not be identified with the
	abstract identity $\ker_z(\mathcal F_{\mathrm{full}})=\ker(d\jmath_z)$ of
	Theorem~\ref{thm:kernel-characterization}: the two kernels belong to two
	different maps, $P$ and $\jmath$. A degeneracy recorded by
	$K_{\mathrm{full}}(\lambda)$ is a redundancy of the finite-mixture
	parametrization, and carries no information about $\ker(d\jmath_z)$, which
	Section~\ref{sec:lifted-limit-space} leaves undetermined.
\end{remark}

\begin{remark}
	A variation may lie in
	$K_{\mathrm{vel}}(\lambda)\setminus K_{\mathrm{flux}}(\lambda),$
	and hence be invisible to velocity observables while remaining detectable by
	quadratic flux observables. Thus the distinction between velocity
	observability and lifted flux observability is essential.
\end{remark}

The ambient stress-gauge variations produced by
Proposition~\ref{prop:mixture-kernels} should be contrasted with their smooth
prototype in Corollary~\ref{cor:smooth-stress-gauge}, where such directions
are abundant and global because $S$ ranges over an unconstrained
linear space. Here the admissible variations are confined to the tangent
space $T_\lambda\Delta_N$ of the simplex, and it is this confinement,
together with the realizability requirement of
Remark~\ref{rem:mixture-ambient-only}, that makes globality a nontrivial
question.

\subsection{Pointwise phase-counting}
\label{subsec:phase-counting}

The counts in this subsection are pointwise in
$(t,x)\in(0,T)\times\mathbb T^d$. They concern algebraic variations of the
weights at a fixed spacetime point. They should not be confused with tangent
vectors to the global finite-dimensional map $P:\Delta_N\to\mathbf E$, where
$\delta\lambda$ is a single coefficient vector independent of $(t,x)$.

Fix
$(t,x)\in(0,T)\times\mathbb T^d$
and write
$a_i:=v_i(t,x)\in\mathbb R^d.$
The pointwise space of weight variations preserving total mass and barycentric
velocity is
$$
K^{(1)}_{t,x}
=
\left\{
\delta\lambda\in\mathbb R^N
\;\middle|\;
\sum_{i=1}^N\delta\lambda_i=0,
\quad
\sum_{i=1}^N\delta\lambda_i a_i=0
\right\}.
$$

\begin{proposition} 
	\label{prop:velocity-invisible-phase-count}
	One has
	$$
	\dim K^{(1)}_{t,x}
	\geq
	\max\{0,N-(d+1)\}.
	$$
	In particular, if
	$$
	N>d+1,
	$$
	then $K^{(1)}_{t,x}$ is nontrivial. For configurations for which the
	corresponding zeroth--first moment map has maximal rank, and in particular
	for generic configurations, equality holds.
\end{proposition}

\begin{proof}
	The constraints defining $K^{(1)}_{t,x}$ are the kernel of
	$$
	L^{(1)}_{t,x}:\mathbb R^N\longrightarrow\mathbb R^{d+1},
	\qquad
	L^{(1)}_{t,x}(\delta\lambda)
	=
	\left(
	\sum_{i=1}^N\delta\lambda_i,
	\sum_{i=1}^N\delta\lambda_i a_i
	\right).
	$$
	Therefore
	$$
	\dim K^{(1)}_{t,x}
	=
	N-\operatorname{rank}L^{(1)}_{t,x}
	\geq
	N-(d+1).
	$$
	Since the dimension is nonnegative, the stated bound follows. If
	$L^{(1)}_{t,x}$ has maximal rank
	$\min\{N,d+1\},$
	then equality holds.
\end{proof}

For exact phases,
$Q_i=v_i\otimes v_i,$
preserving the flux to first order amounts pointwise to preserving the second
moment. Define
$$
K^{(2)}_{t,x}
=
\left\{
\delta\lambda\in\mathbb R^N
\;\middle|\;
\sum_{i=1}^N\delta\lambda_i=0,\quad
\sum_{i=1}^N\delta\lambda_i a_i=0,\quad
\sum_{i=1}^N\delta\lambda_i a_i\otimes a_i=0
\right\}.
$$

\begin{proposition} 
	\label{prop:velocity-flux-invisible-phase-count}
	One has
	$$
	\dim K^{(2)}_{t,x}
	\geq
	\max\left\{
	0,\,
	N-\frac{(d+1)(d+2)}2
	\right\}.
	$$
	For configurations for which the corresponding zeroth--first--second moment
	map has maximal rank, and in particular for generic configurations, equality
	holds.
\end{proposition}

\begin{proof}
	The defining constraints are the kernel of the moment map
	$$
	L^{(2)}_{t,x}:\mathbb R^N
	\longrightarrow
	\mathbb R\oplus\mathbb R^d\oplus\operatorname{Sym}(d),
	$$
	$$
	L^{(2)}_{t,x}(\delta\lambda)
	=
	\left(
	\sum_{i=1}^N\delta\lambda_i,\,
	\sum_{i=1}^N\delta\lambda_i a_i,\,
	\sum_{i=1}^N\delta\lambda_i a_i\otimes a_i
	\right).
	$$
	The target has dimension
	$$
	1+d+\frac{d(d+1)}2
	=
	\frac{(d+1)(d+2)}2.
	$$
	Rank--nullity gives
	$$
	\dim K^{(2)}_{t,x}
	=
	N-\operatorname{rank}L^{(2)}_{t,x}
	\geq
	N-\frac{(d+1)(d+2)}2.
	$$
	Together with nonnegativity of the dimension, this gives the asserted lower
	bound. If $L^{(2)}_{t,x}$ has maximal rank
	$$
	\min\left\{
	N,\frac{(d+1)(d+2)}2
	\right\},
	$$
	then equality holds.
\end{proof}

\begin{remark}
	\label{rem:sym-versus-deviatoric}
	The third component of $L^{(2)}_{t,x}$ takes values in
	$\operatorname{Sym}(d)$, not in $\mathcal S^d_0$, since
	$a_i\otimes a_i$ has trace $|a_i|^2$. This is consistent with the
	trace-free stress convention of Section~\ref{sec:lifted-limit-space}: by
	\eqref{eq:absorbed-mixture} the trace part of the second moment is absorbed
	into the reconstructed pressure $p_{\mathrm{rec}}$, and only the
	deviatoric part enters the stress variable. The counting above is therefore
	performed before this splitting; the corresponding trace-free statement is
	Theorem~\ref{thm:deviatoric-second-moment-realization}.
\end{remark}

The next result is purely pointwise and algebraic. It shows that, once
sufficiently many phases are present at a fixed point
$(t,x)\in(0,T)\times\mathbb T^d,$
one can keep the barycentric velocity fixed while prescribing a first-order
trace-free quadratic flux variation. It does not, by itself, assert the
existence of a global spacetime mixture direction with constant coefficients.

\begin{theorem} 
	\label{thm:deviatoric-second-moment-realization}
	Assume that
	$$
	N\geq\frac{(d+1)(d+2)}2
	$$
	and that
	$a_1,\ldots,a_N\in\mathbb R^d$
	are in generic position in the sense that the moment map
	$$
	L^{(2)}:\mathbb R^N
	\longrightarrow
	\mathbb R\oplus\mathbb R^d\oplus\operatorname{Sym}(d),
	$$
	$$
	L^{(2)}(\delta\lambda)
	=
	\left(
	\sum_{i=1}^N\delta\lambda_i,\,
	\sum_{i=1}^N\delta\lambda_i a_i,\,
	\sum_{i=1}^N\delta\lambda_i a_i\otimes a_i
	\right),
	$$
	has full rank, that is,
	$$
	\operatorname{rank}L^{(2)}
	=
	\min\left\{N,\frac{(d+1)(d+2)}2\right\}
	=
	\frac{(d+1)(d+2)}2 .
	$$
	Then $L^{(2)}$ is surjective, and for every
	$T_0\in\mathcal S^d_0$
	there exists
	$\delta\lambda\in\mathbb R^N$
	such that
	$$
	\sum_{i=1}^N\delta\lambda_i=0,
	\qquad
	\sum_{i=1}^N\delta\lambda_i a_i=0,
	\qquad
	\operatorname{dev}\left(
	\sum_{i=1}^N\delta\lambda_i a_i\otimes a_i
	\right)
	=
	T_0.
	$$
\end{theorem}

\begin{proof}
	Since
	$$
	\dim\left(
	\mathbb R\oplus\mathbb R^d\oplus\operatorname{Sym}(d)
	\right)
	=
	\frac{(d+1)(d+2)}2
	\leq N,
	$$
	the full-rank hypothesis states exactly that the rank of $L^{(2)}$ equals
	the dimension of its target, that is, $L^{(2)}$ is surjective.
	Choose a symmetric matrix $M_0\in\operatorname{Sym}(d)$ with
	$\operatorname{dev}M_0=T_0$; for example, one may take
	$M_0=T_0.$
	By surjectivity there exists
	$\delta\lambda\in\mathbb R^N$
	with
	$L^{(2)}(\delta\lambda)=(0,0,M_0).$
	The first two components give the required zeroth- and first-moment
	constraints, while the third gives
	$$
	\operatorname{dev}\left(
	\sum_{i=1}^N\delta\lambda_i a_i\otimes a_i
	\right)
	=
	\operatorname{dev}M_0
	=
	T_0.
	$$
\end{proof}

\begin{corollary}
	\label{cor:prescribed-deviatoric-stress-variation}
	Let $Q_i=v_i\otimes v_i$, let $\lambda^0\in\Delta_N^\circ$, and fix
	$(t,x)\in(0,T)\times\mathbb T^d$ such that the hypotheses of
	Theorem~\ref{thm:deviatoric-second-moment-realization} hold for
	$a_i=v_i(t,x)$. Then for every $T_0\in\mathcal S^d_0$ there is
	$\delta\lambda\in T_{\lambda^0}\Delta_N$ with
	$\sum_i\delta\lambda_i v_i(t,x)=0$ and
	$$
	\frac{d}{ds}\bigg|_{s=0}
	\operatorname{dev}C[\lambda^0+s\,\delta\lambda](t,x)
	=
	T_0 .
	$$
	Equivalently, by Lemma~\ref{lem:first-variation-dev-C}, the absorbed stress
	$\widetilde R[\lambda]$ of \eqref{eq:absorbed-mixture} admits at $(t,x)$ the
	first variation
	$\sum_i\delta\lambda_i R_i(t,x)+T_0$
	at frozen barycentric velocity. As emphasized in
	Remark~\ref{rem:pointwise-versus-global} below, this is a pointwise
	statement: the vector $\delta\lambda$ produced here depends on $(t,x)$.
\end{corollary}

\begin{proof}
	Combine Theorem~\ref{thm:deviatoric-second-moment-realization} with
	Lemma~\ref{lem:first-variation-dev-C}. Note that
	$\sum_i\delta\lambda_i=0$ ensures
	$\delta\lambda\in T_{\lambda^0}\Delta_N$, and that
	$\lambda^0\in\Delta_N^\circ$ ensures
	$\lambda^0+s\,\delta\lambda\in\Delta_N^\circ$ for $|s|$ small.
\end{proof}

Returning to the finite-dimensional ambient mixture map, let
$\lambda^0\in\Delta_N^\circ$
and define
$\lambda(s):=\lambda^0+s\,\delta\lambda.$
If
$\sum_{i=1}^N\delta\lambda_i=0,$
then
$\lambda(s)\in\Delta_N$
for all sufficiently small $|s|$. Since $\lambda^0$ lies in the relative
interior, in fact
$\lambda(s)\in\Delta_N^\circ$
for sufficiently small $|s|$. The path
$s\mapsto P(\lambda(s))$
is affine, hence smooth, in the lifted ambient space.

If
$\sum_{i=1}^N\delta\lambda_i v_i=0,$
then its velocity component has zero first variation. If, in addition,
$\sum_{i=1}^N\delta\lambda_iQ_i=0$
and
$\sum_{i=1}^N\delta\lambda_iR_i\neq0,$
then
$$
\frac{d}{ds}\Big|_{s=0}P(\lambda(s))
=
\left(
0,0,\sum_{i=1}^N\delta\lambda_iR_i
\right)
$$
is a nonzero stress-gauge direction in the lifted ambient model. Membership of
this path in
$\mathbf X(v_0,\mathcal A)$
is an additional admissibility and realizability question; it does not follow
from the finite-dimensional ambient mixture construction alone.

The preceding phase counts can be reformulated in terms of moment observables
on finite atomic measures. Let
$$
\nu_\lambda:=\sum_{i=1}^N\lambda_i\delta_{a_i}
$$
be the atomic measure associated with
$a_1,\ldots,a_N\in\mathbb R^d.$
For a multi-index $\alpha$, define
$$
M_\alpha(\nu_\lambda)
:=
\int_{\mathbb R^d}y^\alpha\,d\nu_\lambda(y)
=
\sum_{i=1}^N\lambda_i a_i^\alpha.
$$

\begin{theorem} 
	\label{thm:minimal-observables}
	Let
	$$
	T_\lambda\Delta_N
	=
	\left\{
	\delta\lambda\in\mathbb R^N
	\;\middle|\;
	\sum_{i=1}^N\delta\lambda_i=0
	\right\}.
	$$
	For generic distinct points
	$a_1,\ldots,a_N\in\mathbb R^d,$
	the polynomial moment observables of order at most $k$ separate all
	directions in $T_\lambda\Delta_N$ if and only if
	$$
	\binom{d+k}{d}\geq N.
	$$
	Consequently, the minimal polynomial order sufficient to separate all weight
	directions generically is
	$$
	k^*(N,d)
	:=
	\min\left\{
	k\in\mathbb N_0:
	\binom{d+k}{d}\geq N
	\right\}.
	$$
	In particular, moments of order at most two separate all weight directions
	generically if and only if
	$$
	N\leq\frac{(d+1)(d+2)}2.
	$$
\end{theorem}

\begin{proof}
	Let
	$\mathcal P_{\leq k}(\mathbb R^d)$
	denote the space of scalar polynomials of degree at most $k$. Its
	dimension is
	$$
	m_k:=\dim\mathcal P_{\leq k}(\mathbb R^d)
	=
	\binom{d+k}{d}.
	$$
	The moment observables of order at most $k$ are represented by the
	evaluation map
	$$
	E_k:\mathbb R^N\longrightarrow
	\mathcal P_{\leq k}(\mathbb R^d)^*,
	\qquad
	E_k(\delta\lambda)(p)
	=
	\sum_{i=1}^N\delta\lambda_i p(a_i).
	$$
	The constant polynomial gives
	$E_k(\delta\lambda)(1)
	=
	\sum_{i=1}^N\delta\lambda_i,$
	so the restriction of $E_k$ to
	$T_\lambda\Delta_N$
	detects precisely the nonconstant moment variations.
	For generic distinct points
	$a_1,\ldots,a_N\in\mathbb R^d,$
	the polynomial evaluation map $E_k$ has maximal rank
	$\min\{m_k,N\}.$
	If
	$m_k\geq N,$
	then $E_k$ is generically injective on $\mathbb R^N$, and hence its
	restriction to $T_\lambda\Delta_N$ is injective.
	Conversely, if
	$m_k<N,$
	then, even including the constant polynomial, the full evaluation map has
	rank at most $m_k<N$. On the codimension-one space
	$T_\lambda\Delta_N,$
	the nonconstant moments provide at most $m_k-1$ independent constraints.
	Since
	$$
	\dim T_\lambda\Delta_N=N-1>m_k-1,
	$$
	the restricted moment map has a nontrivial kernel generically. Therefore the
	moment observables of order at most $k$ do not separate all tangent weight
	directions.
	
	Thus separation holds generically if and only if
	$$
	m_k=\binom{d+k}{d}\geq N.
	$$
	The formula for $k^*(N,d)$ follows immediately. For $k=2$,
	$$
	\binom{d+2}{d}
	=
	\frac{(d+1)(d+2)}2,
	$$
	which gives the final assertion.
\end{proof}

\begin{remark}
	The quantity
	$$
	\frac{(d+1)(d+2)}2
	$$
	is the dimension of the space of scalar polynomials of degree at most two in
	$d$ variables. Below or at this threshold, generic zeroth-, first-, and
	second-moment observables separate all finite-mixture weight directions.
	Above it, higher-order moments are generically required to separate all such
	directions.
\end{remark}

\begin{remark} 
	\label{rem:pointwise-versus-global}
	The pointwise spaces
	$K^{(1)}_{t,x}$
	and 
	$K^{(2)}_{t,x}$
	allow the coefficients $\delta\lambda_i$ to be chosen separately at each
	$(t,x)\in(0,T)\times\mathbb T^d.$
	By contrast, a tangent vector to the finite-dimensional mixture map
	$P:\Delta_N\to\mathbf E$
	is determined by a single coefficient vector
	$\delta\lambda\in T_\lambda\Delta_N.$
	Therefore pointwise phase-counting provides local algebraic feasibility
	conditions, but does not by itself imply the existence of a global spacetime
	mixture direction with constant coefficients. Additional compatibility,
	localization, or admissible realization arguments are required for such a
	conclusion.
\end{remark}

\subsection{Realizability of mixture paths in the lifted limit space}
\label{subsec:mixture-realizability}

By Remark~\ref{rem:mixture-ambient-only} the mixture map $P$ is a priori only a
map into the lifted ambient space $\mathbf E$. This subsection gives an
explicit sufficient criterion under which a stress-absorbing
modification of $P$ takes values in $\mathbf X(v_0,\mathcal A)$.

The relevant geometric object is the joint graph
\begin{equation}
	\label{eq:joint-graph}
	\Gamma_{\mathcal A}(t,x)
	=
	\bigl\{(u,S)\in\mathbb R^d\times\mathcal S^d_0
	\mid
	S\in\mathcal A(t,x,u)\bigr\}.
\end{equation}
Fiberwise convexity of $\mathcal A(t,x,\cdot)$ does not imply convexity
of $\Gamma_{\mathcal A}(t,x)$, since the fibers may move nonconvexly with $u$;
joint convexity is therefore an independent hypothesis below. The following
elementary lemma is the only convexity input we need.

\begin{lemma}
	\label{lem:interior-radius-concave}
	Let $\Gamma\subseteq\mathbb R^m$ be closed and convex, and define the
	interior radius
	$$
	\rho_\Gamma(y)
	:=
	\sup\bigl\{r\geq0:\overline B_r(y)\subseteq\Gamma\bigr\}
	\in[0,+\infty],
	\qquad y\in\mathbb R^m,
	$$
	with the convention $\rho_\Gamma\equiv+\infty$ if $\Gamma=\mathbb R^m$.
	Assume $\Gamma\neq\mathbb R^m$. Then
	$\rho_\Gamma=\operatorname{dist}(\cdot,\mathbb R^m\setminus\Gamma)$ is
	finite and vanishes on
	$\mathbb R^m\setminus\operatorname{int}\Gamma$, and:
	\begin{enumerate}
		\item\label{it:irc-concave}
		$\rho_\Gamma$ is concave: for all $y_1,\ldots,y_N\in\mathbb R^m$ and
		$\lambda\in\Delta_N$,
		$$
		\rho_\Gamma\Bigl(\sum_{i=1}^N\lambda_iy_i\Bigr)
		\geq
		\sum_{i=1}^N\lambda_i\,\rho_\Gamma(y_i);
		$$
		\item\label{it:irc-lipschitz}
		$\rho_\Gamma$ is $1$-Lipschitz; in particular
		$\rho_\Gamma(y+z)\geq\rho_\Gamma(y)-|z|$ for all $y,z$;
		\item\label{it:irc-fiber}
		if $\Gamma=\Gamma_{\mathcal A}(t,x)$ as in \eqref{eq:joint-graph},
		$y=(u,S)$ and $\rho_\Gamma(y)\geq r$, then
		$$
		\bigl\{S'\in\mathcal S^d_0:|S'-S|\leq r\bigr\}
		\subseteq\mathcal A(t,x,u),
		$$
		hence $S\in\operatorname{int}\mathcal A(t,x,u)$ as soon as $r>0$, the
		interior being relative to $\mathcal S^d_0$.
	\end{enumerate}
\end{lemma}

\begin{proof}
	\eqref{it:irc-concave}: put $r_i:=\rho_\Gamma(y_i)$ and let
	$y:=\sum_i\lambda_iy_i$, $r:=\sum_i\lambda_ir_i$. If some $r_i=0$ the claim
	is weaker, so assume $r>0$ and fix $z$ with $|z|\leq r$. Writing
	$z=\sum_i\lambda_i\frac{r_i}{r}z\cdot\frac{r}{r}$ more explicitly, set
	$z_i:=\frac{r_i}{r}\,\frac{r}{|z|}\,z$ if $z\neq0$ and $z_i:=0$ otherwise;
	then $|z_i|\leq r_i$, so $y_i+z_i\in\Gamma$, and
	$\sum_i\lambda_i(y_i+z_i)=y+z$ because
	$\sum_i\lambda_i z_i=\frac{r}{|z|}\,\frac{\sum_i\lambda_ir_i}{r}\,z=z$
	when $z\neq0$. Convexity of $\Gamma$ gives $y+z\in\Gamma$, i.e.
	$\overline B_r(y)\subseteq\Gamma$ and $\rho_\Gamma(y)\geq r$.
	
	\eqref{it:irc-lipschitz}: the distance function to the nonempty set
	$\mathbb R^m\setminus\Gamma$ is $1$-Lipschitz, and it coincides with
	$\rho_\Gamma$ since $\overline B_r(y)\subseteq\Gamma$ if and only if
	$\operatorname{dist}(y,\mathbb R^m\setminus\Gamma)\geq r$ for closed
	$\Gamma$.
	
	\eqref{it:irc-fiber}: for $|S'-S|\leq r$ the point $(u,S')$ lies in
	$\overline B_r\bigl((u,S)\bigr)\subseteq\Gamma_{\mathcal A}(t,x)$, which by
	\eqref{eq:joint-graph} means $S'\in\mathcal A(t,x,u)$.
\end{proof}

\begin{proposition}
	\label{prop:exact-mixture-absorbed-realizability}
	Let $x_i=\Lambda(v_i,p_i,R_i)$, $1\leq i\leq N$, be exact lifted states
	arising from triples in
	$\mathfrak S^\infty_{\mathrm{str}}(v_0,\mathcal A)$ with common initial
	datum $v_0$, and assume:
	\begin{enumerate}
		\item\label{it:emar-convex}
		for every $(t,x)\in(0,T)\times\mathbb T^d$ the joint graph
		$\Gamma_{\mathcal A}(t,x)$ of \eqref{eq:joint-graph} is closed and
		convex;
		\item\label{it:emar-uniform}
		the phases are jointly uniformly admissible, i.e.
		$$
		\varepsilon
		:=
		\inf_{(t,x)\in(0,T)\times\mathbb T^d}
		\;\min_{1\leq i\leq N}
		\rho_{\Gamma_{\mathcal A}(t,x)}
		\bigl(v_i(t,x),R_i(t,x)\bigr)
		\;>\;0;
		$$
		\item\label{it:emar-budget}
		the phases are close in the covariance budget
		$$
		\delta
		:=
		\frac12
		\max_{1\leq i,j\leq N}
		\bigl\|v_i-v_j\bigr\|_{L^\infty((0,T)\times\mathbb T^d)}^2
		\;<\;\varepsilon .
		$$
	\end{enumerate}
	Then for every $\lambda\in\Delta_N$ the absorbed mixture state
	$$
	\Lambda\bigl(
	v[\lambda],\,p_{\mathrm{rec}}[\lambda],\,\widetilde R[\lambda]
	\bigr)
	=
	\bigl(
	v[\lambda],\,v[\lambda]\otimes v[\lambda],\,
	R[\lambda]+\operatorname{dev}C[\lambda]
	\bigr)
	$$
	of \eqref{eq:absorbed-mixture} belongs to $\mathbf X(v_0,\mathcal A)$ and
	lies on the exact Euler locus, with the quantitative margin
	\begin{equation}
		\label{eq:absorbed-margin}
		\overline B_{\varepsilon-\delta}
		\bigl(v[\lambda](t,x),\widetilde R[\lambda](t,x)\bigr)
		\subseteq\Gamma_{\mathcal A}(t,x)
		\qquad
		\text{for every }(t,x)\in(0,T)\times\mathbb T^d .
	\end{equation}
	Moreover
	$\lambda\mapsto
	\Lambda\bigl(v[\lambda],p_{\mathrm{rec}}[\lambda],\widetilde R[\lambda]\bigr)$
	is smooth on $\Delta_N$ as a map into $\mathbf E$ with image in
	$\mathbf X(v_0,\mathcal A)$, and for $\lambda\in\Delta_N^\circ$ and
	$\delta\lambda\in T_\lambda\Delta_N$ with
	$\sum_{i=1}^N\delta\lambda_i\,v_i=0$ its differential is
	$$
	\left(
	0,\,0,\,
	\sum_{i=1}^N\delta\lambda_iR_i
	+
	\operatorname{dev}\Bigl(\sum_{i=1}^N\delta\lambda_i\,v_i\otimes v_i\Bigr)
	\right).
	$$
\end{proposition}

\begin{proof}
	Fix $\lambda\in\Delta_N$ and abbreviate
	$\widetilde x[\lambda]:=
	\Lambda\bigl(v[\lambda],p_{\mathrm{rec}}[\lambda],\widetilde R[\lambda]\bigr)$.
	We verify the four items of
	Definition~\ref{def:lift-smooth-subsolution} for the triple
	$\bigl(v[\lambda],p_{\mathrm{rec}}[\lambda],\widetilde R[\lambda]\bigr)$.
	The maps $v[\lambda]$, $R[\lambda]$,
	$p[\lambda]$ are convex combinations of smooth data and $C[\lambda]$ is a
	quadratic expression in them by \eqref{eq:mixture-covariance-defect}, so
	all entries lie in $C^\infty([0,T]\times\mathbb T^d)$, and
	$\widetilde R[\lambda]$ takes values in $\mathcal S^d_0$ by
	\eqref{eq:absorbed-mixture}; this is item~(\ref{it:lss-regularity}). Since
	$\sum_i\lambda_i=1$ and $v_i(0,\cdot)=v_0$ for every $i$, we get
	$v[\lambda](0,\cdot)=v_0$, which is item~(\ref{it:lss-initial}).
	By Lemma~\ref{lem:mixture-covariance-sign} the absorbed
	triple satisfies
	$$
	\partial_tv[\lambda]
	+\operatorname{div}\bigl(v[\lambda]\otimes v[\lambda]\bigr)
	+\nabla p_{\mathrm{rec}}[\lambda]
	=
	-\operatorname{div}\widetilde R[\lambda],
	\qquad
	\operatorname{div}v[\lambda]=0,
	$$
	the trace part of the covariance defect having been absorbed into
	$p_{\mathrm{rec}}[\lambda]$ and its deviatoric part into
	$\widetilde R[\lambda]$. All entries being smooth on
	$[0,T]\times\mathbb T^d$, the identity holds pointwise there, which is
	item~(\ref{it:lss-equation}). In particular the flux equals
	$v[\lambda]\otimes v[\lambda]$ exactly, so $\widetilde x[\lambda]$ lies on
	the exact Euler locus.
	
	Fix $(t,x)\in(0,T)\times\mathbb T^d$ and write
	$\rho:=\rho_{\Gamma_{\mathcal A}(t,x)}$. By
	Lemma~\ref{lem:mixture-covariance-sign} one has
	$\bigl|\operatorname{dev}C[\lambda](t,x)\bigr|\leq\delta$ for a.e.\
	$(t,x)$, hence for every $(t,x)$ by continuity of the $v_i$. Using
	hypothesis~(\ref{it:emar-convex}) together with
	Lemma~\ref{lem:interior-radius-concave}\eqref{it:irc-lipschitz} and
	\eqref{it:irc-concave} and hypothesis~(\ref{it:emar-uniform}),
	$$
	\rho\bigl(v[\lambda],\widetilde R[\lambda]\bigr)
	\;\geq\;
	\rho\bigl(v[\lambda],R[\lambda]\bigr)-\delta
	\;\geq\;
	\sum_{i=1}^N\lambda_i\,\rho\bigl(v_i,R_i\bigr)-\delta
	\;\geq\;
	\varepsilon-\delta
	\;>\;0
	$$
	at $(t,x)$, where hypothesis~(\ref{it:emar-budget}) was used in the last
	step. This is exactly \eqref{eq:absorbed-margin}, and
	Lemma~\ref{lem:interior-radius-concave}\eqref{it:irc-fiber} converts it
	into
	$$
	\widetilde R[\lambda](t,x)
	\in
	\operatorname{int}\mathcal A\bigl(t,x,v[\lambda](t,x)\bigr),
	$$
	which is item~(\ref{it:lss-admissibility}). Hence
	$\bigl(v[\lambda],p_{\mathrm{rec}}[\lambda],\widetilde R[\lambda]\bigr)
	\in\mathfrak S^\infty_{\mathrm{str}}(v_0,\mathcal A)$ and
	$\widetilde x[\lambda]\in\mathbf X(v_0,\mathcal A)$.
	
	The three components of
	$\widetilde x[\lambda]$ are polynomial in $\lambda$ with coefficients
	smooth in $(t,x)$, whence smoothness on $\Delta_N$. For
	$\lambda\in\Delta_N^\circ$ and $\delta\lambda\in T_\lambda\Delta_N$ with
	$\sum_i\delta\lambda_iv_i=0$ we get
	$\mathrm d v[\lambda](\delta\lambda)=0$, hence
	$\mathrm d\bigl(v[\lambda]\otimes v[\lambda]\bigr)(\delta\lambda)=0$, and
	Lemma~\ref{lem:first-variation-dev-C} gives
	$\mathrm d\bigl(\operatorname{dev}C[\lambda]\bigr)(\delta\lambda)
	=\operatorname{dev}\bigl(\sum_i\delta\lambda_i\,v_i\otimes v_i\bigr)$,
	while $\mathrm dR[\lambda](\delta\lambda)=\sum_i\delta\lambda_iR_i$. The
	stated formula follows, and it is a genuine tangent direction in
	$\mathbf X(v_0,\mathcal A)$ because the whole curve stays in
	$\mathbf X(v_0,\mathcal A)$ by the previous steps.
\end{proof}

\begin{remark}
	\label{rem:joint-radius-nondegeneracy}
	Hypothesis~(\ref{it:emar-uniform}) is strictly stronger than
	item~(\ref{it:lss-admissibility}) of
	Definition~\ref{def:lift-smooth-subsolution}. The latter only gives
	$R_i(t,x)\in\operatorname{int}\mathcal A(t,x,v_i(t,x))$ fiberwise,
	with no uniformity in $(t,x)$ and no control transversal to the fiber.
	Hypothesis~(\ref{it:emar-uniform}) in addition forces
	$\Gamma_{\mathcal A}(t,x)$ to have nonempty interior in
	$\mathbb R^d\times\mathcal S^d_0$, a genuine structural restriction: if the
	fibers depend on $u$ so sharply that no joint ball fits, then
	$\rho_{\Gamma_{\mathcal A}}\equiv0$ and the proposition is vacuous. It is
	precisely this transversal room that finances the absorption of
	$\operatorname{dev}C[\lambda]$, whose size is controlled only in
	$\mathcal S^d_0$ but which is added at a shifted velocity
	$v[\lambda]\neq v_i$.
\end{remark}

\begin{remark}
	\label{rem:two-mechanisms}
	Two distinct mechanisms are at play. Barycentric mixing at frozen flux
	keeps $Q$ as an independent variable and produces the ambient point
	$P(\lambda)$, which by Remark~\ref{rem:mixture-ambient-only} need not lie
	in $\mathbf X(v_0,\mathcal A)$ without a characterization of the latter.
	Proposition~\ref{prop:exact-mixture-absorbed-realizability} instead pays
	the quadratic covariance defect into the stress and the reconstructed
	pressure, remains on the exact Euler locus, and is obtained directly
	from the lifting map together with the geometry of
	$\Gamma_{\mathcal A}$; the price is the $\varepsilon$-budget
	hypothesis~\ref{it:emar-budget}, which quantifies how much velocity
	oscillation the admissibility multifunction can absorb. Combining
	Corollary~\ref{cor:prescribed-deviatoric-stress-variation} with
	Proposition~\ref{prop:exact-mixture-absorbed-realizability} shows that,
	within this budget, prescribed trace-free first-order stress variations at
	frozen barycentric velocity are attainable pointwise.
\end{remark}

\begin{remark}
	\label{rem:realizability-open}
	When $\Gamma_{\mathcal A}(t,x)$ fails to be jointly convex, or when the
	interior radius degenerates as in
	Remark~\ref{rem:joint-radius-nondegeneracy}, the conclusion of
	Proposition~\ref{prop:exact-mixture-absorbed-realizability} may fail and
	the ambient image $P(\Delta_N)$ need not meet
	$\mathbf X(v_0,\mathcal A)$ at all. The fundamental open problem is to
	characterize, for a given admissibility multifunction $\mathcal A$, the
	largest subset $\Delta_N^{\mathrm{real}}\subseteq\Delta_N$ such that the
	absorbed mixture lies in $\mathbf X(v_0,\mathcal A)$ for every
	$\lambda\in\Delta_N^{\mathrm{real}}$. Hypotheses
	\ref{it:emar-convex}--\ref{it:emar-budget} give an explicit but
	$\lambda$-independent sufficient condition; a sharp answer would have to
	trade the local geometry of $\Gamma_{\mathcal A}$ against the
	$\lambda$-dependent covariance $C[\lambda]$, and would yield sharp
	tangency conditions for the lifted limit space together with a description
	of the Euler locus inside $\mathbf X(v_0,\mathcal A)$.
\end{remark}

\vskip 12pt

\paragraph{\bf Data availability statement} No data are available for this work.

\vskip 12pt

\paragraph{\bf Conflict of interest statement} The authors declare no conflict of interest.

\vskip 12pt

\paragraph{\bf Funding} No funding supported this work.

\vskip 12pt

\paragraph{\bf Acknowledgements} A.A. and B.D. would like to thank Yazd University for its support. This work was carried out while the first author was a postdoctoral researcher at Yazd University. 
J.-P.M. thanks the France 2030 framework programme Centre Henri Lebesgue ANR-11-LABX-0020-01
for creating an attractive mathematical environment.

\vskip 12pt

\paragraph{\bf Author's Note on AI Assistance}
Portions of the text were developed with the assistance of a generative language model (OpenAI ChatGPT). The AI was used to assist with drafting, editing, and standardizing the bibliography format. All mathematical content, structure, and theoretical constructions were provided, verified, and curated by the authors. The authors assume full responsibility for the correctness, originality, and scholarly integrity of the final manuscript.

\end{document}